\documentclass{article}

\usepackage{arxiv}
\usepackage{geometry}
\usepackage{graphicx} 
\usepackage{amsmath} 
\usepackage{amsthm}   
\usepackage{amsfonts}   
\usepackage{amssymb}  
\usepackage{latexsym}  
\usepackage{float} 
\usepackage{verbatim}
\usepackage{appendix}
\usepackage{float}     
\usepackage{epstopdf}   
\usepackage{color}      
\usepackage{mathrsfs}
\usepackage{hyperref}
\usepackage{indentfirst}
\usepackage{enumerate}
\allowdisplaybreaks[4]
\numberwithin{equation}{section}

\renewcommand\appendix{\par
    \setcounter{section}{0}
    \gdef\thesection{Appendix~ \Alph{section}}
    \setcounter{equation}{0}
     \renewcommand\theequation{A.~\arabic{equation}}} 

\newcommand{\PP}{\mathbb{P}}
\newcommand{\RR}{\mathbb{R}}

\theoremstyle{definition}
\newtheorem{theorem}{Theorem}[section]
\newtheorem{lemma}[theorem]{Lemma}

\newtheorem{condition}[theorem]{Condition}

\newtheorem{definition}[theorem]{Definition}
\newtheorem{example}{Example}

\newtheorem{assumption}[theorem]{Assumption}

\newtheorem{remark}[theorem]{Remark}

\title{Strong convergence in the infinite horizon of numerical methods for stochastic delay differential equations}

\author{
 Yudong Wang and Hongjiong Tian \\
 Department of Mathematics, Shanghai Normal University, Shanghai, 200234, China\\
\texttt{  wangyudong\_edu@163.com; hjtian@shnu.edu.cn}
}

\begin{document}
\maketitle
\begin{abstract}
    In this work, we present a general technique for establishing the strong convergence of numerical methods for stochastic delay differential equations (SDDEs) in the infinite horizon. This technique can also be extended to analyze certain continuous function-valued segment processes associated with the numerical methods, facilitating the numerical approximation of invariant measures of SDDEs. To illustrate the application of these results, we specifically investigate the backward and truncated Euler-Maruyama methods. Several numerical experiments are provided to demonstrate the theoretical results.
\end{abstract}

\section{Introduction}
Due to the widespread occurrence of time delays and environmental noises in real-world systems, stochastic delay differential equations (SDDEs) have emerged as a crucial mathematical model in the realms of science and engineering in recent decades \cite{AHMG2007,  CGMP2019, LM2006, M2007, M1974}. However, obtaining explicit solutions for SDDEs is challenging, leading to an increasing focus on numerical approximations in recent years.

For numerical methods, a critical focus lies in studying their strong convergence. Many good results have been obtained, including but not limited to: the classic Euler-Maruyama method \cite{Buckwar2000, MS2003}, the backward Euler-Maruyama method \cite{LCF2004,Zhou2015}, the truncated Euler-Maruyama method \cite{GMY2018, SHGL2022}, the tamed Euler method \cite{JY2017}, etc. However, due to the presence of time delay, the solutions of SDDEs are not Markov processes, and therefore, properties based on Markov processes do not apply. At this point, our focus should shift to the segment process associated with the solution — a continuous function-valued Markov process \cite{M1986}. Therefore,  it is sometimes necessary to analyze the convergence of the segment process associated with the numerical method \cite{BSY-2023, LMS2023}.

More specifically, for the solution of SDDE, $\{x(t)\}_{t\geq -\tau}$, and its corresponding continuous numerical solution, $\{X(t)\}_{t \geq -\tau}$, let $x_t:=\{x(t+\theta): \theta \in [-\tau, 0]\}$ and $X_t:=\{X(t+\theta): \theta \in [-\tau, 0]\}$ represent the segment processes associated with these solutions. The classical study of strong convergence aims to find some upper bounds for the difference between $x(t)$ and $X(t)$, or between $x_t$ and $X_t$, i.e., for some positive constants $T, p$ and $q$
\begin{align*}
    \sup \limits_{0 \leq t \leq T}\mathbb{E}|x(t)-X(t)|^p \leq C_T \Delta^q  \quad\text{or}\quad \sup \limits_{0 \leq t \leq T}\mathbb{E}\|x_t-X_t\|^p \leq C_T \Delta^q,
\end{align*}
where $\Delta$ is the step size and $C_T$ is a constant dependent on $T$. In general, due to the application of tools like the Gronwall lemma, the typical form of $C_T$ is an exponential function related to time, such as $C_T=ce^{cT}$ for some positive constant $c$ \cite{LCF2004, SHGL2022}, which implies that $C_T$ is a function that grows over time. 

However, this poses some problems: one issue is that when studying the long-term behavior of the system \cite{ CKBW2015,DNY2016,LM2009}, it is obvious that $C_T$ will grow large as $T \to \infty$. Consequently, to ensure that the upper bound of the convergence error remains within the tolerable range, it is theoretically necessary to adjust the step size $\Delta$ to be correspondingly very small. With a large time span and a small step size, this clearly makes the computation expensive or even unacceptable. The second is dealing with certain issues that are difficult to determine within a fixed time interval, such as the timing of investment opportunities based on favorable market conditions \cite{DGM2021,GA1997}, presents challenges since these opportunities may arise at any moment. Therefore, any assumption that presupposes a specific time interval, including conducting convergence analysis within a finite time, is not appropriate.

In actual numerical simulations, we find that for some SDDEs and numerical methods, the actual error does not always increase with time after fixing the appropriate step size. Instead, it appears that the error has a time-independent upper bound but depends on the step size.  This intrigues us and leads to the following questions: 1) Does such a time-independent upper bound on the error truly exist, and if so, what is its concrete form?  2) Which types of SDDEs and numerical methods exhibit such an error upper bound?

Thus, unlike most existing results, this paper is devoted to  the study of the strong convergence of numerical methods in infinite horizon, that is, to try to obtain a constant $C$ independent of $T$, such that
\begin{align}\label{UIT}
    \sup \limits_{t \geq 0}\mathbb{E}|x(t)-X(t)|^p \leq C \Delta^q  \quad\text{or}\quad \sup \limits_{t \geq 0 }\mathbb{E}\|x_t-X_t\|^p \leq C \Delta^q.
\end{align}
If \eqref{UIT} holds, One of the simplest and most straightforward applications is that we no longer need to adjust the step size to meet the specified threshold accuracy when numerically simulating the behavior of the system over a long period of time. In this paper, we introduce a technique that can be employed to establish strong convergence in the infinite horizon. More importantly, the techniques we present are general and can be extended to many one-step methods beyond the backward Euler-Maruyama (BEM) and truncated Euler-Maruyama (TEM) methods mentioned in this article. Concurrently, we note that progress has been made in this direction by Crisan et al. \cite{ACO2023, CDO2021}, who have also developed a method for analyzing the convergence of various numerical methods in the infinite horizon. However, unlike theirs, our focus is on the strong convergence of stochastic differential delay equations in infinite horizons, rather than on weak convergence.

Now, let's informally outline this technique. Firstly, we partition the infinite time interval into a countably finite time intervals of length $T$. In each finite time interval, we construct a new true solution starting from the numerical solution, so that the original error can be divided into two parts for analysis. Then with the help of four preconditions, we find that one part is decaying, while the other part is always bounded. Consequently, even if the time tends to infinity, the error between the original numerical solution and the true solution does not explode but instead has an upper bound, which is positively related to the step size. It's worth emphasizing that the aforementioned four preconditions have been extensively studied\cite{  GLL2023, LMS2023,WWM2019, YZM2003}. This implies a wealth of available results for establishing the strong convergence of various numerical methods in the infinite horizon. In addition,
a series of works by Mao et al (see e.g. \cite{HMS2003,Mao2007, Mao2015}) must be thanked for inspiring this paper.

Now, by applying our new results, we can answer our previous two questions to some extent.  Specifically, the SDDEs and the numerical methods that satisfy our four preconditions exhibit an error upper bound that is independent of time and can be quantified.

This paper is organized as follows: In Section 2, we present three main results, namely, a general technique that can be used to establish the strong convergence of numerical solutions or their corresponding segment processes, and the usefulness of this technique in numerical approximations to invariant measures of the true solution.
In Sections 3 and 4, to illustrate the application of this technique, we analyze the strong convergences of the numerical solution of the truncated Euler-Maruyama (TEM) method and the segment process associated with the backward Euler-Maruyama (BEM) method in the infinite horizon, respectively. The numerical experiments are presented in Section 5.

\section{Mathematical Preliminaries}
At the beginning of this section, let's introduce some necessary notations. 
Denote by $|\cdot|$ the Euclidean norm in $\RR^d$ and the trace
norm  in $\RR^{d\times m}$, and by $\langle \cdot,\cdot\rangle$ the inner product in $\RR^{d}$.
For real numbers $a$ and $b$,  let $a\vee{}b=\max\{a,b\}$ and $a\wedge{}b=\min\{a,b\}$, respectively.
Let $\lfloor a\rfloor$ be the integer part of the real
number $a$. Let  $\boldsymbol{1}_A(x)$ be the indicator function of the set $A$.
 Let $\mathbb{R}_{+}=[0,\infty)$. Denote by { $C:=C([-\tau,~0];~\mathbb{R}^{d})$}  the  family of continuous functions $X$ from
$[-\tau,~0]$ to $\mathbb{R}^{d}$ with the supremum norm $\|X\|
=\sup_{-\tau\leq\theta\leq0}| X(\theta)|$. 
For $p\geq2$,  denote by $\mathcal C_{\mathcal F_0}^{p}:=\mathcal C_{\mathcal F_0}^{p}([-\tau,~0];~ \RR^d)$  the  family of all $\mathcal F_0$-measurable bounded $C$-valued random variables $\xi= \{\xi(\theta): -\tau \leq \theta \leq 0\}$.

This paper focuses on the scalar SDDE of the form
\begin{align}\label{SDDE}
dx(t)=f(x(t),x(t-\tau))dt +g(x(t),x(t-\tau))dW(t),~~~t> 0,
\end{align}
with the initial data $\xi=\{\xi(\theta)$: $\theta\in [-\tau, 0]\} \in C$,
where  $\tau>0$,  $f:  \RR^d \times \RR^d \rightarrow \RR^{d}$ and $g: \RR^d\times \RR^d  \rightarrow \RR^{d\times m}$ are Borel measurable, $W(t)$ is an m-dimensional Brownian motion on a probability space  $( \Omega,~\mathcal{F}, ~\PP )$ with a right-continuous complete filtration $\{{\mathcal{F}}_{t}\}_{t\geq 0}$, and $\xi$ is an $\mathcal F_0$-measurable, continuous function-valued random variable from
$[-\tau,~0]$ to $\mathbb{R}^{d}$.
Let $\{x_{t}\}_{t\geq 0} $ be the  segment process, where $x_{t}(\theta):=x(t+\theta)$ for $\theta\in[-\tau,0]$, and sometimes to emphasize the initial value $\xi$ at $t=0$, we also write the solution of \eqref{SDDE} as $x(t; 0, \xi)$  while the corresponding $x_t(\theta)$ is written as $x_t^{0, \xi}(\theta)$. It must be mentioned that the process $\{x_t\}_{t \geq 0}$ associated with the true solution of $\eqref{SDDE}$ is a homogeneous Markov process; see, for example, Theorem 1.1 in \cite{M1986} or Proposition 3.4 in \cite{RRV2006} for more details. 

For simplicity, for given $T, T_1, T_2, \tau \in \mathbb{R}_{+}$,  we may suppose without loss of generality that $\Delta=\frac{T}{N}=\frac{T_1}{N_1}=\frac{T_2}{N_2}=\frac{\tau}{M}$ with some integers $N$, $N_1$, $N_2$ and $M$, and then let $t_k = k\Delta$ for $k \geq -M$.

Moreover, to explore the invariant measure of the solution, some necessary concepts also need to be introduced. Let $\mathfrak{B}(C)$ be the Borel algebra of $C$ and $\mathcal{P}(C):=\mathcal{P}(C, \mathfrak{B}(C))$ be the family of  probability measures on $(C,\mathfrak{B}(C))$. For any $\xi \in C$ and $B \in \mathfrak{B}(C)$, the transition probabilities for the segment processes $x_t$ with initial value $x_0 =\xi$ is defined as
\begin{align*}
    \PP_t(\xi, B) :=\PP(x_t \in B | x_0 \in \xi) 
\end{align*}
Define metric $d_{\mathbb{L}}$ on $\mathcal{P}(C)$ by
\begin{equation*}
    \mathit{d}_{\mathbb{L}}(\mathbb{P}_1 , \mathbb{P}_2) = \sup \limits_{\mathit{F} \in \mathbb{L}} \left| \int_{C} \mathit{F}(x) \mathbb{P}_1(dx) - \int_{C} \mathit{F}(x)\mathbb{P}_2(dx) \right|,
\end{equation*}
where $\mathbb{L}$ is the test functional space
\begin{align*}
    \mathbb{L} := \Big\{F: C \to \mathbb{R} \Big| |F(x)-F(y)| \leq |x-y| \quad\text{and}\quad \sup \limits_{x\in C} |F(x)| \leq 1 \Big\}
\end{align*}
The weak convergence of probability measures can be illustrated in terms of metric $\mathit{d}_{\mathbb{L}}$ \cite{IW1989}. That is, a sequence of probability measures $\{\mathbb{P}_t\}_{t \geq 0}$ in $\mathcal{P}(C)$ converge weakly to a probability measure $\mathbb{P} \in \mathcal{P}(C)$ if and only if
\begin{equation*}
    \lim \limits_{ k \to \infty} \mathit{d}_{\mathbb{L}}(\mathbb{P}_k, \mathbb{P}) = 0.
\end{equation*}
Then we define the invariant measure by using the concept of weak convergence.

\begin{definition}
 For any initial value $\xi \in \mathbb{R}^d$, the $C$-valued process $\{x_t^{0, \xi}\}_{t \geq 0}$ is said to have a invariant measure $\pi \in \mathcal{P}(C)$ if the transition probability measure ${\mathbb{P}}_t(\xi,\cdot)$ converges weakly to $\pi( \cdot)$ as $t \to \infty$ for every $\xi \in C$, that is
    \begin{equation*}
        \lim \limits_{t \to \infty} \left(\sup \limits_{\mathit{F} \in \mathbb{L}} \left|\mathbb{E}(\mathit{F}(x_t^{0, \xi}))-\mathbb{E}_{\pi}(\mathit{F})\right|\right) = 0,
    \end{equation*}
    where
    \begin{equation*}
        \mathbb{E}_{\pi}(\mathit{F}) = \int_{C} \mathit{F}(y)\pi(dy).
    \end{equation*}
\end{definition}

 \section{Main theorems}
In this section, we will present our three main results. To better reflect the similarities and differences between the strong convergence analysis of the numerical solution itself and the corresponding segment process, we will state the preconditions and theorems first, and put the concrete proof behind. 
The first result is about the strong convergence of the numerical solution in the infinite horizon. We state the preconditions as follows:

\begin{condition}\label{con1}~
     \begin{enumerate}
\item[(i)] There exists a continuous numerical solution, denoted as $\{X(t; 0, \xi)\}_{t \geq -\tau}$, which can be used to approximate the solution of \eqref{SDDE}, and the segment process $\{X_{t_k}^{0, \xi}\}_{k \in \mathbb{N}}$ corresponding to $\{X(t; 0, \xi)\}_{t \geq -\tau}$ is a time-homogeneous Markov chain.

\item[(ii)] The numerical solution is moment bounded the infinite horizon, i.e., for any $t \geq -\tau$, there exists a positive constant $M_{1}$ satisfying
\begin{align*}
    \mathbb{E}\Big(|X(t; 0, \xi)|^r\Big) \leq M_{1}.
\end{align*}
where $r$ is a positive constant.

\item[(iii)] The true solutions of \eqref{SDDE} are asymptotically attractive, i.e., for some $p \in \mathbb{R}_+$, and any compact subset $K \subset C$, any two true solutions with different initial values satisfy
\begin{align*}
    \mathbb{E}\Big(|x(t; 0, \xi) - x(t;0, \eta)|^p\Big) \leq M_{2}\Big(\sup \limits_{-\tau \leq s \leq 0} \mathbb{E}\big(|\xi(s)-\eta(s)|^p\big)\Big)e^{-M_{3}t},
\end{align*}
    where $M_{2}$, $M_{3}$ are a pair positive constants and $\xi,\eta \in K \times K$.

\item[(iii)]  For any $t \in [T_1, T_2]$ and $\xi \in C_{\mathcal F_0}^{2}$, note that $x(t; T_1, \xi)$ is the true solution of Eq. \eqref{SDDE} with initial value $\xi$  at $t=T_1$. Then  the moment of the error between $x(t;T_1,\xi)$ and $X(t;T_1,\xi)$ satisfies
\begin{align*}
  \sup \limits_{T_1 \leq t \leq T_2}  \mathbb{E}\Big(|x(t;T_1,\xi)- X(t;T_1,\xi)|^p\Big) \leq C_{T_2-T_1} \Big(1+ \sup \limits_{ -\tau\leq t \leq 0}\mathbb{E}\big(|\xi(t)|^r\big)\Big) \Delta^{q},
\end{align*}
where $q, T_1, T_2 \in \mathbb{R}$, $T_2 >T_1$, $C_{T_2-T_1}$ is a constant dependent on $T_2-T_1$ and $r$ is consistent with (ii).

\end{enumerate}
\end{condition}

  \begin{remark}
      \begin{enumerate}
\item[(1)] The four conditions in Condition \ref{con1} may seem abstract, but in fact, the first condition holds for almost all one-step numerical methods. The third condition is related to the true solution and is easy to obtain. The fourth condition in Condition \ref{con1} typically only needs to be shown to hold on the time interval $[0.T]$ for any $T\geq \tau$, that is, to show that the classical finite-time strong convergence holds. Then, by the time-homogeneity of the solutions, it can be generalized to any finite time interval with minor modifications. 

The crucial consideration lies in the third condition —— the moment boundedness of the continuous approximation $\{X(t;0,\xi)\}_{t \geq -\tau}$. This condition is used to ensure that the right-hand side of the inequality in the fourth condition is bounded, thus playing an indispensable role in our proof. In some cases, the third condition can be weakened to the moment bounded of the true solution; refer to Theorem \ref{T-22} for details.

\item[(2)]For more details about the asymptotic attraction and moment boundedness of the solutions, please see e.g. \cite{GLL2023,LMYY2018,LMY2019,YM2003, YM2004,YZM2003}.

\end{enumerate}
  \end{remark}

The following theorem is proved in Appendix A.

\begin{theorem}\label{T--2}
 Suppose Condition \ref{con1} holds, then the numerical solution $\{X(t; 0, \xi)\}_{t \geq -\tau}$ converges strongly to the underlying solution of \eqref{SDDE} in the infinite horizon, i.e., for any $t \geq -\tau$
 \begin{displaymath}
       \mathbb{E}\Big(|x(t; 0, \xi)-X(t; 0, \xi)|^p\Big) \leq C \Delta^{q},
 \end{displaymath}
 where $C$ is a positive constant independent of T.
\end{theorem}

From Theorem \ref{T--2}, We know that under Condition \ref{con1},
the numerical solution is strongly convergent in the infinite horizon. However, as we mentioned earlier, sometimes we need to shift our focus to the segment process associated with the solution due to the delay. Obviously, at this point, Condition \ref{con1} no longer applies to convergence analysis of the numerical segment process. We need to introduce some new preconditions, similar to condition 2.1 but related to the segment processes rather than the solutions themselves. Let us now state them as follows:

\begin{condition}\label{con19} ~
\begin{enumerate}
\item[(i)] There exists a continuous numerical solution, denoted as $\{X(t; 0, \xi)\}_{t \geq -\tau}$, which can be used to approximate the solution of \eqref{SDDE}, and the numerical segment process $\{X_{t_k}^{0, \xi}\}_{k \in \mathbb{N}}$ corresponding to $\{X(t; 0, \xi)\}_{t \geq -\tau}$ is a time-homogeneous Markov process.

\item[(ii)] For any $t \geq 0$, the true segment process $\{x_t^{0, \xi}\}_{t \geq 0}$ is moment bounded in the following form
    \begin{equation*}
         \mathbb{E} \|x_t^{0, \xi} \|^p \leq K_{{1}} \mathbb{E} \|\xi \|^p e^{-K_{{2}} t} + K_{{3}},
    \end{equation*}
where $K_{{1}}, K_{{2}}, K_{{3}}$ are positive constants.

\item[(iii)]The segment processes associated with \eqref{SDDE} are asymptotic attractive, i.e., for any compact subset $K \subset  C$ and some $p \in \mathbb{R}_+$, ~any two segment processes corresponding to the true solution with different initial values satisfy
\begin{equation*}
    \mathbb{E}\Big(\|x_{t}^{0,\xi}-x_t^{0,\eta} \|^p \Big) \leq K_{{4}}\mathbb{E}\Big( \|\xi -\eta\|^p \Big)e^{-K_{{5}} t},
    \end{equation*}
    where $K_{{4}}$, $K_{{5}}$ are a pair positive constants and $\xi,\eta \in K \times K$.

    \item[(iv)]  For any $t_k \in [T_1, T_2]$, the moment of the error between the true and numerical segment processes satisfies
    \begin{align*}
         \sup \limits_{T_1 \leq t_k \leq T_2} \mathbb{E}\Big(\|x_{t_k}^{T_1, X_{T_1}^{0, \xi}}-X_{t_k}^{T_1, X_{T_1}^{0, \xi}}\|^p \Big) \leq C_{T_2-T_1}\Big(1+\sup \limits_{ T_1-\tau\leq t_k \leq T_1}\mathbb{E}\|X_{t_k}^{0,\xi}\|^p\Big)\Delta^q
    \end{align*}
where $q, T_1,T_2\in \mathbb{R}$,$T_1 <T_2$, $C_{T_2-T_1}$ is a constant dependent on $T_2-T_2$, and we set $\sup \limits_{ -\tau\leq t_k \leq 0}\mathbb{E}\|X_{t_k}^{0,\xi}\|^p:=\mathbb{E}\|\xi\|^p$.  

\end{enumerate}
  \end{condition}

    Comparing Conditions \ref{con1} and \ref{con19}, we can see that the first condition does not change, while the third and fourth conditions seem to simply replace the solutions themselves with the corresponding segment processes. The most significant difference lies in the second condition, where the moment bounded is not related to the numerical solution but to the segment process associated with the true solution. However, it should be noted that in Condition \ref{con19}, consistency in the order of moments on both sides of the inequality is required, while in Condition \ref{con1}, inconsistency is permissible. This is because if the order of the moments in the fourth condition is consistent, it can be used to derive the moment bounded of the numerical solution from that of the true solution, for more details please see Lemma \ref{L13}, and in fact this idea also applies to the solution itself. 

This means that there are actually two versions of both conditions \ref{con1} and \ref{con19}. The first version allows differing orders of the moments on either side of the inequality in the fourth condition but requires the boundedness of moments for the numerical solution or its corresponding segment process. The second version requires only that the true solution or its corresponding segment process moments are bounded, yet the fourth condition requires consistency in the order of moments on both sides of the inequality. Thus, we can choose the appropriate version according to our needs to analyze the strong convergence of numerical methods in infinite horizons.

Then, in order to prove the strong convergence of the numerical solution in the infinite horizon, the following lemma needs to be established first, which is also an independent new result.

     \begin{lemma}\label{L13}
    Suppose that the second and fourth conditions in Condition \ref{con19} hold, then for any $k \geq 0$ and $\xi \in \mathcal C_{\mathcal F_0}^{p}$, there exists a positive constant $\Delta^*$ such that for any $\Delta \in (0, \Delta^*)$, the process $\{X_{t_k}^{0,\xi}\}_{k \geq 0}$ satisyes
    \begin{equation*}
         \mathbb{E} \|X_{t_k}^{0,\xi}\|^2 \leq K_6 \mathbb{E} \|\xi \|^2 + K_7,
    \end{equation*}
    where $K_6 $ and $ K_7$ are positive constants .
    \end{lemma}

The lemma is proved in Appendix B. Now, we can state the second result of this paper.

\begin{theorem}\label{T-22}
    Under Condition \ref{con19}, for any $k \in \mathbb{N}$ and $\xi \in \mathcal C_{\mathcal F_0}^{p}$, the numerical segment process $\{X_{t_k}^{0,\xi}\}_{k \in \mathbb{N}}$ converges strongly to the true segment process $\{x_{t_k}^{0,\xi}\}_{k \in \mathbb{N}}$ in the infinite horizon, that is, 
        \begin{align*}
        \sup \limits_{k \in \mathbb{N}}  \mathbb{E}\Big(\|x_{t_k}^{0,\xi}-X_{t_k}^{0, \xi}\|^2\Big) \leq C \Delta^q,
    \end{align*}
    where $C$ is a positive constant independent of $T$.
\end{theorem}
With the help of Lemma \ref{L13}, this theorem can be proved in a similar way to Theorem \ref{T--2}, so we omit the concrete proof.

Next, utilizing the strong convergence of $\{X_{t_k}^{0,\xi}\}_{k \geq 0}$, we can show that the  probability measure of $\{X_{t_k}^{0,\xi}\}_{k \geq 0}$ converges to the invariant measure of the segment process associated with the SDDE \eqref{SDDE}.
For any $\xi \in C$ and $B \in \mathfrak{B}(C)$, the transition probabilities for the segment process $X_{t_k}$ with initial value $X_0 =\xi$ is defined as
\begin{align*}
   \overline{\PP}_{t_k}(\xi, B) :=\PP(X_{t_k} \in B | X_0 = \xi)
\end{align*}

Following a similar argument as in \cite[Theorem 3.1]{YZM2003}, it is not difficult to see that the process $\{x_t\}_{t\geq 0}$ has a unique invariant measure denoted by $\pi (\cdot)$ under the second and fifth conditions of Condition \ref{con19}.

 Then, we have the following theorem.
\begin{theorem}\label{T-9}
    Under Condition \ref{con19}, for any $\xi \in C_{\mathcal F_0}^{2}$, the probability measure of $\{X_{t_k}^{0,\xi}\}_{k \geq 0}$ converges to the underlying invariant measure $\pi(\cdot)$ in the Fortet-Mourier distance $d_{\mathbb{L}}$ as the step size tends to zero, that is
   \begin{align*}
        \lim \limits_{\substack{k\to \infty \\ \Delta \to 0}}d_{\mathbb{L}}(\overline{\mathbb{P}}_{t_k}(\xi, \cdot), \pi(\cdot))=0.
   \end{align*}
\end{theorem}
\begin{proof}
    First, recalling that under Condition \ref{con19}, the process $\{x_t\}_{t \geq 0}$ has a unique invariant measure $\pi$. This means that for any $\epsilon >0$, there exists a $T>0$ such that for any $t_k \geq T$
    \begin{align*}
          d_{\mathbb{L}}(\mathbb{P}_{t_k}(\xi, \cdot), \pi(\cdot)) \leq \frac{\epsilon}{2}.
    \end{align*}
    Next, from the definitions of ~$\mathbb{L}$ and $d_{\mathbb{L}}$, for any $F \in \mathbb{L}$, we have
    \begin{align*}
          d_{\mathbb{L}}(\mathbb{P}_{t_k}(\xi, \cdot), \overline{\mathbb{P}}_{t_k}(\xi, \cdot)) = \sup \limits_{F \in \mathbb{L}}|\mathbb{E}(F(x_{t_k}^{0,\xi}))-\mathbb{E}(F(X_{t_k}^{0,\xi}))| \leq \mathbb{E}(2 \wedge \|x_{t_k}^{0,\xi}- X_{t_k}^{0, \xi}\|).
    \end{align*}
    Note that $\xi \in C_{\mathcal F_0}^{2}$, which means that $\mathbb{E}\|\xi\|^2 \leq \infty$. Thus, by Theorem \ref{T-22} and the H$\ddot{o}$lder inequality, there exists a sufficiently small $\Delta^*$ such that for any $\Delta \in (0, \Delta^*)$
    \begin{align*}
          d_{\mathbb{L}}(\mathbb{P}_{t_k}(\xi, \cdot), \overline{\mathbb{P}}_{t_k}(\xi, \cdot)) \leq \mathbb{E}(\|x_{t_k}^{0,\xi}- X_{t_k}^{0, \xi}\|) \leq \Big(\mathbb{E}(\|x_{t_k}^{0,\xi}- X_{t_k}^{0, \xi}\|^2)\Big)^{\frac{1}{2}} \leq \frac{\epsilon}{2}
    \end{align*}
   Finally, by the triangle inequality yields
    \begin{align*}
      d_{\mathbb{L}}(\mathbb{P}_{t_k}(\xi, \cdot), \pi(\cdot))    \leq d_{\mathbb{L}}(\mathbb{P}_{t_k}(\xi, \cdot), \overline{\mathbb{P}}_{t_k}(\xi, \cdot))+d_{\mathbb{L}}(\mathbb{P}_{t_k}(\xi, \cdot), \pi(\cdot)).
    \end{align*}
    The proof is complete.
\end{proof}
The most classical way to study the numerical approximation of the invariant measures of SDEs or SDDEs is to first prove the existence of invariant measures for both the numerical and true solutions or their corresponding segment processes, and then prove that the two invariant measures will converge as the step size decreases (see e.g. \cite{GLL2023,LMY2019,YM2004}). 

As can be seen from the above proof, in this paper, it is no longer necessary to directly prove the existence of the invariant measure of the numerical segment process. Instead, we use the strong convergence to prove that the probability measure of the numerical segment process converges to the probability measure of the true segment process in the infinite horizon. Subsequently, if the invariant measure of the true segment process exists, indicating convergence of the probability measure of the true segment process to the invariant measure, then the probability measure of the numerical segment process also converges to this invariant measure. 

Therefore, we simplify the proof process. Specifically, for example, we do not have to prove that the numerical solution or its corresponding segment process is also asymptotically attractive,  as in \cite{GLL2023,LMY2019,YM2004}, in order to establish the existence of numerical invariant measures.

\section{Truncated Euler-Maruyama method}
The theory established in the previous section shows that to establish the convergence of numerical methods in infinite perspectives, the numerical and true solutions need to satisfy a series of properties described in Condition \ref{con1}. The question is, do such solutions exist? We will provide a positive answer by considering the truncated Euler-Maruyama method under the following assumptions.
\begin{assumption}\label{as1-1}
    There exist a pair of positive constants $q, a_1$ such that for any  $x,~\bar{x},~y,~\bar y\in \RR ^d $
    \begin{align}\label{S3-1-1}
        |f(x,y)- f(\bar{x}, \Bar{y})|  &\leq a_1(|x-\Bar{x}|+|y-\Bar{y}| )(1+|x|^q + |\Bar{x}|^q+ |y|^q+ |\Bar{y}|^q), \\
        |g(x,y)- g(\bar{x}, \Bar{y})|^2 &\leq a_1(|x-\Bar{x}|^2+|y-\Bar{y}|^2 )(1+|x|^q + |\Bar{x}|^q+ |y|^q+ |\Bar{y}|^q)
    \end{align}
\end{assumption}
\begin{assumption}\label{as1-2}
    There exist nonnegative constants $p$, $b_1, b_2, b_3, \overline{b}_1, \overline{b}_2, \overline{b}_3$ with $p > 4q \vee 3q+8, b_2 > b_3, ~\overline{b}_2> \overline{b}_3$ such that for any $x,~y\in \RR ^d $
    \begin{align}
    \notag
       & \Big(1+|x|^2\Big)^{\frac{p}{2}-1}\Big(\langle 2x, f(x,y) \rangle + (p-1)|g(x,y)|^2\Big) \leq b_1 - b_2|x|^p +b_3|y|^p. \\ \notag
     &\quad\quad\quad  \langle 2x, f(x,y) \rangle + (p-1)|g(x,y)|^2 \leq \overline{b}_1 -\overline{ b}_2|x|^2 +\overline{b}_3|y|^2. 
    \end{align}
\end{assumption}
\begin{assumption}\label{as1-3}
    There exist nonnegative constants $ b_4, b_5, \sigma$ with $b_4 > b_5$ and $\sigma >1$ such that for any  $x,~\bar{x},~y,~\bar y\in \RR ^d $
    \begin{align}\label{S3-1-3}
        &\quad 2\langle x-\Bar{x}, f(x,y)- f(\Bar{x}, \Bar{y}) \rangle + \sigma|g(x,y) - g(\Bar{x}, \Bar{y})|^2 \leq -b_4|x-\Bar{x}|^2 +b_5|y- \Bar{y}|^2.
    \end{align}
\end{assumption}
\begin{assumption}\label{as1-4}
     The initial value is H$\rm{\ddot{o}}$lder continuous, i.e., there exist a positive constant $K_{{4}}$ such that for all $-\tau \leq s < t \leq 0$
     \begin{align}\label{S3-1-4}
         \mathbb{E}|\xi(t)-\xi(s)|^p \leq K_{{4}}(t-s)^{\frac{p}{2}}.
     \end{align}
\end{assumption}

Then, we proceed to introduce the definition of the truncated EM method. Choose a strictly increasing function $\phi : [1, \infty) \to \mathbb{R}_+$ such that $\phi (r) \to \infty$ as $r \to \infty$ and
\begin{align}\label{S3-1-5}
    \sup \limits_{|x|\wedge |y| \leq r} \left( \frac{|f(x,y)|}{1+|x|} \wedge \frac{|g(x,y)|^2}{(1+|x|)^2}\right) \leq \phi(r), \quad \quad \forall ~r \geq 1.
\end{align}
For some given $\Delta^* \in (0,1)$ and $N,M \in \mathbb{N}_+$, let $\Delta = \frac{\tau}{M} = \frac{T}{N} \in (0, \Delta^*)$  and $\phi^{-1} : [\phi(1), + \infty) \to \mathbb{R}_+$ be the inverse function of $\phi$. Define the truncation mapping $\pi_{\Delta}: \mathbb{R}^d \to \mathbb{R}^d$ by 
\begin{align}\label{PI}
    \pi_{\Delta}(x) = \Big(|x| \wedge \phi^{-1}(K \Delta^{- \nu}) \Big) \frac{x}{|x|},
\end{align}
where $\frac{x}{|x|} = 0$ when $x =0 $, $\nu :=\frac{q+2}{p-2} \in (0, \frac{1}{3}]$, and 
$K := 1 \vee \phi(1) \vee f(0,0) \vee (g(0,0))^2$. 
From \eqref{S3-1-1}, we note that
\begin{align}
    |f(x,y)|&\leq |f(x,y) - f(0,0)| + |f(0,0)| \\ \notag
    &\leq a_1 (|x|+|y|)(1+|x|^q+|y|^q)+|f(0,0)| \\ \notag
    &\leq C_1 (1+|x|^{q+1}+|y|^{q+1}),
    \end{align}
    and
   \begin{align}\label{PG}
    |g(x,y)|& \leq C_2 (1+|x|^{\frac{q}{2}+1}+|y|^{\frac{q}{2}+1}),
\end{align}
where $C_1:=3a_1+ |f(0,0)|$ and $C_2= 3\sqrt{a_1}+|g(0,0)| $. 
It also implies that we can take
\begin{align}
    \phi (r) &= [C_1(1+2r^{q+1})] \wedge [C_2(1+2r^{\frac{q}{2}+1})]^2 \\ \notag
    &\leq C_1+2C_2^2 + (2C_1+8C_2^2)r^{q+2},
\end{align}
where  $r \geq 1$. Thus
\begin{align}\label{PHI-1}
    \phi ^{-1} (y) = \Big(\frac{y- (C_1+2C_2^2)}{2C_1+8C_2^2}\Big)^{\frac{1}{q+2}},
\end{align}
for $y \geq 3C_1+10C_2^2$.
Then for any $\xi \in C_{\mathcal F_0}^{p}$, define the truncated EM method of SDDE \eqref{SDDE} by
\begin{align}\label{DTSDDE}
    \begin{split}
  \left \{
 \begin{array}{ll}
y_k=\widetilde{y}_k  = \xi (t_k), \quad \quad k=-M,...,0,\\
 \widetilde{y}_k = y_{k-1} + f(y_{k-1}, y_{k-1-M})\Delta+ f(y_{k-1}, y_{k-1-M})\Delta W_{k-1},\quad i=1,2,...,\\
  y_{k}= \pi_{\Delta}(\widetilde{y}_k) \quad \quad i=1,2,...,
 \end{array}
 \right.
 \end{split}
 \end{align}
 where $t_k = k\Delta$ for $k \leq -M$ and $\Delta W_k=W(t_{k+1})-W(t_k)$. This, along with \eqref{S3-1-5} and \eqref{PI}, implies that
 \begin{align}\label{CON}
     |f(y_k, y_{k-M})| \leq  K \Delta^{- \nu}(1+|y_k|), \quad |g(y_k, y_{k-M})| \leq  K^{\frac{1}{2}} \Delta^{- \frac{\nu}{2}}(1+|y_k|).
 \end{align}
 Finally, define several continuous - time numerical schemes $\widetilde{y}(t), y(t), z(t)$ by
 \begin{align}\label{CDTSDDE2}
     \widetilde{y}(t) :=\widetilde{y}_k, \quad  \quad y(t) :=y_k, \quad \quad \forall t \in [t_k, t_{k+1}),
 \end{align}
 as well as
 \begin{align}\label{CDTSDDE}
    \begin{split}
  \left \{
 \begin{array}{ll}
z(t)=\xi(t), \quad\quad\quad\quad\forall t \in [-\tau,0],\\
z(t)=y_k +f(y_k, y_{k-M}) (t-t_k)+g(y_k, y_{k-M}) (W(t)-W(t_k)), \quad \forall t \in [t_k, t_{k+1}).
 \end{array}
 \right.
 \end{split}
 \end{align}
 Obviously, $y(t_k) = z(t_k) = y_k$ for $i \geq -M$. Denote $Z_t:= \{z(t+\theta): -\tau \leq \theta \leq 0\}$, then in a way similar to the proof of \cite[Lemma 5.1]{BSY-2023} , we obtain the following result, which is the first condition in condition \ref{con1}.
 \begin{lemma}
 The process $\{Z_{t_k}\}_{k \geq 0}$ is a homogeneous Markov chain, i.e., for any $B \in \mathfrak{B}(C)$ and $\xi \in C$
  \begin{align}
      \mathbb{P}\big(~Z_{t_{k+1}} \in B |~Z_{t_k} = \xi  \big)= \mathbb{P}\big(~Z_{t_1} \in B |~Z_{0}= \xi \big),
  \end{align}
  and 
  \begin{align}
       \mathbb{P}\big(~Z_{t_{k+1}} \in B |~\mathcal{F}_{t_k} \big)= \mathbb{P}\big(~Z_{t_{k+1}} \in B |~Z_{t_k} \big).
  \end{align}
 \end{lemma}

 \subsection{The second condition in Condition \ref{con1}}
 In this subsection, we will show the boundedness of moments of $z(t)$.
 \begin{lemma}\label{L23}
     Under Assumption \ref{as1-2}, the truncated Euler-Maruyama numerical solution \eqref{CDTSDDE} satisfies
     \begin{align*}
         \sup \limits_{t \geq 0} \mathbb{E} |z(t)|^p \leq C_3 \Big(  1+\sup \limits_{ -M \leq i \leq 0}\mathbb{E}\Big(|y_{i}|^p\Big) \Big) 
     \end{align*}
     where $C_3$ is a positive constant.
 \end{lemma}
   \begin{proof}
       Following a similar way to the proof of Theorem 3.1 in \cite{SHGL2022}, we get the following result
       \begin{align}\label{L46-EQ-1}
         &\quad  \mathbb{E}\Big( (1+|\widetilde{y}_k|^2)^{\frac{p}{2}} \Big| ~\mathcal{F}_{t_{k-1}}\Big)  \\ \notag
         & \leq \Big(1+|y_{k-1}|^2\Big)^{\frac{p}{2}} \Big(1+C\Delta^{\frac{4}{3}} + \frac{p\Delta}{2} \frac{(1+|y_{k-1}|^2)(2 \langle y_{k-1}, f_{k-1} \rangle + |g_{k-1}|^2) + (p-2)| \langle y_{k-1}, g_{k-1} \rangle|^2}{(1+ |y_{k-1}|^2)^2} \Big) \\ \notag
         &\leq \Big(1+C\Delta^{\frac{4}{3}}\Big)\Big(1+|y_{k-1}|^2\Big)^{\frac{p}{2}} + \frac{p\Delta}{2} \Big(1+|y_{k-1}|^2\Big)^{\frac{p}{2} -1} \Big(2 \langle y_{k-1}, f_{k-1} \rangle +(p-1) |g_{k-1}|^2\Big) \\ \notag
         &\leq \Big(1+C\Delta^{\frac{4}{3}}\Big)\Big(1+|y_{k-1}|^2\Big)^{\frac{p}{2}} + \frac{p\Delta}{2} \Big(b_1 - b_2 |y_{k-1}|^p + b_3 |\widetilde{y}_{k-1-M}|^p\Big),
       \end{align}
      where $f_k := f(y_k, y_{k-M})$ and $g_k :=g(y_k, y_{k-M})$. 
      Next, since $b_2 > b_3$, there exists a unique positive constant $\kappa_1$ satisfying the following inequalities
      \begin{align*}
          b_2-2^{\frac{p}{2}+1}\kappa_1 =b_3 e^{2\kappa_1 \tau} ,
      \end{align*}
     and for any $\kappa \in [0, \kappa_1 \wedge \frac{b_2-b_3}{2} \wedge 1]$
      \begin{align}\label{L23-1}
           b_2-2^{\frac{p}{2}+1}\kappa-b_3 e^{2\kappa\tau} \geq 0 .
      \end{align}
        Choose $\Delta$ sufficiently small such that  $C\Delta ^ {\frac{1}{3}} \leq \kappa$,  then for any fixed $\kappa \in [0,\kappa_1 \wedge \frac{b_2-b_3}{2} \wedge 1]$, by the elementary inequality  $ (|a|-1)^{\frac{p}{2}}\geq |a|^{\frac{p}{2}}/2^{\frac{p}{2}}-1$ with $|a|\geq 1$ and \eqref{L46-EQ-1}, we have
       \begin{align*}
           &\quad  \mathbb{E}\Big( (1+|y_k|^2)^{\frac{p}{2}} \Big| ~\mathcal{F}_{t_{k-1}}\Big) \\ \notag
          & \leq  \mathbb{E}\Big( (1+|\widetilde{y}_k|^2)^{\frac{p}{2}} \Big| ~\mathcal{F}_{t_{k-1}}\Big)  \\ \notag 
           &\leq \Big(1+C\Delta^{\frac{4}{3}}\Big)\Big(1+|y_{k-1}|^2\Big)^{\frac{p}{2}} + \frac{p\Delta}{2} \Big(b_1 -2\kappa (1+|y_{k-1}|^2)^{\frac{p}{2}} +2^{\frac{p}{2}+1} \kappa - (b_2-2^{\frac{p}{2}+1}\kappa) |y_{k-1}|^p \\ \notag
           & \quad+ b_3 |y_{k-1-M}|^p  \Big) \\ \notag
           & \leq \Big(1 - \kappa \Delta\Big) \Big(1+|y_{k-1}|^2\Big)^{\frac{p}{2}} + \frac{p\Delta}{2} \Big(b_1  + 2^{\frac{p}{2}+1} \kappa - (b_2 - 2^{\frac{p}{2}+1}\kappa) |y_{k-1}|^p + b_3 |y_{k-1-M}|^p  \Big).
       \end{align*}
      This implies that      
      \begin{align}\label{L23-2}
       &\quad \mathbb{E}\Big( 1+|y_k|^2\Big)^{\frac{p}{2}}  \\ \notag 
           & \leq \Big(1 - \kappa \Delta\Big)^k \mathbb{E}\Big(1+|y_{0}|^2\Big)^{\frac{p}{2}} + \frac{p\Delta}{2} \Big(b_1  + 2^{\frac{p}{2}+1} \kappa\Big) \sum \limits_{ i=0}^{k-1}(1 - \kappa \Delta)^i  \\ \notag
            & \quad +\frac{p\Delta}{2}\Big( - (b_2 - 2^{\frac{p}{2}+1}\kappa) \sum \limits_{ i=0}^{k-1}(1 - \kappa \Delta)^i \mathbb{E}\Big(|y_{k-i-1}|^p\Big) + b_3 \sum \limits_{ i=0}^{k-1}(1 - \kappa \Delta)^i \mathbb{E}\Big(|y_{k-i-1-M}|^p\Big) \Big).
       \end{align}
       Meanwhile, we note that
       \begin{align}\label{L23-3}
          &\quad \sum \limits_{ i=0}^{k-1}(1 - \kappa \Delta)^i \mathbb{E}\Big(|y_{k-i-1-M}|^p\Big) \\ \notag
           &= (1 - \kappa \Delta)^{-M} \sum \limits_{ i=0}^{k-1}(1 - \kappa \Delta)^{i+M} \mathbb{E}\Big(|y_{k-i-1-M}|^p\Big)  \\ \notag
           & =- (1 - \kappa \Delta)^{-M} \sum \limits_{ i=0}^{M-1}(1 - \kappa \Delta)^{i} \mathbb{E}\Big(|y_{k-i-1}|^p\Big)  + (1 - \kappa \Delta)^{-M} \sum \limits_{ i=0}^{k-1}(1 - \kappa \Delta)^{i} \mathbb{E}\Big(|y_{k-i-1}|^p\Big)\\ \notag
          &\quad + (1 - \kappa \Delta)^{-M} \sum \limits_{ i=k}^{k-1+M}(1 - \kappa \Delta)^{i} \mathbb{E}\Big(|y_{k-i-1}|^p\Big) .
       \end{align}
        Since
        \begin{align*}
            \lim \limits_{\Delta \to 0}   (1-\kappa \Delta)^{-M} =  \lim \limits_{\Delta \to 0}   (1-\kappa \Delta)^{-\frac{\tau}{\Delta}}=e^{\kappa\tau},
        \end{align*}
       there exists a positive constant $\Delta^* \in [0,1)$ such that for any $\Delta \in [0, \Delta^*)$
       \begin{align}\label{L23-5}
           (1-\kappa \Delta)^{-M} \leq e^{2\kappa\tau}.
       \end{align}
      Then, combining \eqref{L23-1}, \eqref{L23-2}, \eqref{L23-3}, and \eqref{L23-5} together, and using the fact $ \sum \limits_{ i=0}^{+\infty}(1 - \kappa \Delta)^i \leq \frac{1}{\kappa \Delta}$,  we obtain
       \begin{align}\label{L23-6}
            &\quad \mathbb{E}\Big( |y_k|^p\Big) \\ \notag
            &\leq \mathbb{E}\Big( 1+|y_k|^2\Big)^{\frac{p}{2}}  \\ \notag 
              & \leq \Big(1 - \kappa \Delta\Big)^k \mathbb{E}\Big(1+|y_{0}|^2\Big)^{\frac{p}{2}} + \frac{p\Delta}{2} \Big(b_1  + 2^{\frac{p}{2}+1} \kappa\Big) \sum \limits_{ i=0}^{k-1}(1 - \kappa \Delta)^i  \\ \notag
            & \quad +\frac{p\Delta}{2}\Big(  b_3(1 - \kappa \Delta)^{-M} \sum \limits_{ i=k}^{k-1+M}(1 - \kappa \Delta)^{i} \mathbb{E}\big(|y_{k-i-1}|^p\big) -b_3 (1 - \kappa \Delta)^{-M} \sum \limits_{ i=0}^{M-1}(1 - \kappa \Delta)^{i} \mathbb{E}\big(|y_{k-i-1}|^p\big) \Big) \\ \notag
           &\leq \mathbb{E}\Big(1+|y_{0}|^2\Big)^{\frac{p}{2}} + \frac{p}{2\kappa} \Big(b_1  +2^{\frac{p}{2}+1} \kappa\Big)  +\frac{pb_3e^{2\kappa \tau}}{2\kappa}\Big(   \sup \limits_{ -M \leq i \leq -1}\mathbb{E}\big(|y_{i}|^p\big) \Big) 
       \end{align}
       On the other hand, by \eqref{PG}, \eqref{DTSDDE}, \eqref{CON} , \eqref{CDTSDDE2}, \eqref{CDTSDDE} and the fact $\mathbb{E}(|W(t)-W(s)|^i) \leq C(i) |t-s|^{\frac{i}{2}}$, for any $t \in [t_k, t_{k+1}) \subset [0, +\infty)$, we can derive that
       \begin{align*}
           & \quad \mathbb{E}\Big(|z(t) - y(t)|^p\Big) =  \mathbb{E}\Big(|z(t) - y_k|^p\Big)  \\ \notag
           &\leq 2^p\mathbb{E} |f(y_k, y_{k-M})|^p \Delta^p + 2^p \mathbb{E}\Big(|g(y_k, y_{k-M})|^p|W(t) - W(t_k)|^p\Big) \\ \notag
           & \leq C_4 \Big(\mathbb{E} |f(y_k, y_{k-M})|^p \Delta^p + \mathbb{E} |g(y_k, y_{k-M})|^p \Delta^{\frac{p}{2}}\Big) \\ \notag
           &\leq C_4 \Big(2^p K^p \Delta^{-\upsilon p}(1+\mathbb{E}|y_k|^p)\Delta^p  + 2^pK^{\frac{p}{2}}\Delta^{-\frac{p\upsilon}{2}}(1+\mathbb{E}|y_k|^p)\Delta^{\frac{p}{2}}\Big) \
       \end{align*}
       where $C(i)$ is a positive constant dependent on $i$ and $C_4$ is positive constant dependent on $p$. 
       As for any $t\in [t_k,t_{k+1}) \subset [-\tau,0]$, it follows from \eqref{S3-1-4} that
       \begin{align*}
            & \quad \mathbb{E}\Big(|z(t) - y(t)|^p\Big)= \mathbb{E}\Big(|\xi(t)-\xi(t_k)|^p\Big) \leq K_1\Delta^{\frac{p}{2}},
       \end{align*}
       Thus, for any $t\in [t_k,t_{k+1}) \subset [-\tau,+\infty)$, we have
       \begin{align}\label{L23-7}
           \mathbb{E}\Big(|z(t) - y(t)|^p\Big) \leq C_5 (1+\mathbb{E}|y_k|^p) \Delta^{\frac{p(1-\upsilon)}{2}},
       \end{align}
       where $C_5$ is a positive constant dependent on $C_4$ and $ K_1$.
       Then, by the elementary inequality $(a+b)^p \leq 2^p(|a|^p + |b|^p)$, we have
       \begin{align}\label{L23-8}
           \mathbb{E}\Big(|z(t)|^p\Big) &= \mathbb{E}\Big(|z(t) -y(t) + y(t)|^p\Big) \\ \notag
           &\leq 2^p  \mathbb{E}\Big(|z(t) -y(t) |^p\Big) + 2^p \mathbb{E}\Big(|y(t)|^p\Big) \\ \notag
           &\leq 2^p C_5 \Delta^{\frac{p(1-\upsilon)}{2}} + 2^p\Big(C_5\Delta^{\frac{p(1-\upsilon)}{2}}+1\Big)\mathbb{E}\Big(|y(t)|^p\Big)
       \end{align}
      Inserting \eqref{L23-6} into \eqref{L23-8}, and let 
      \begin{align*}
          C_3 : &= \Big(2^p C_5 \Delta^{\frac{p(1-\upsilon)}{2}} + 2^p\big(C_5\Delta^{\frac{p(1-\upsilon)}{2}}+1\big)\big(2^{\frac{p}{2}}+ \frac{p}{2\kappa}(b_1+2^{\frac{p}{2}+1} \kappa)\big)\Big)\vee  \Big( 2^p C_5 \Delta^{\frac{p(1-\upsilon)}{2}} \\ \notag
          &\quad+ 2^p\big(C_5\Delta^{\frac{p(1-\upsilon)}{2}}+1\big)\big(2^{\frac{p}{2}} + \frac{pb_3e^{2\kappa \tau}}{2\kappa}
          \big)\Big),   
      \end{align*}
      the desired assertion follows.
   \end{proof}  

\subsection{The third and fourth conditions in Condition \ref{con1}}
In this subsection, we will proceed to prove that the TEM method satisfies the third and fourth conditions in Condition \ref{con1}. First, we need a lemma.
\begin{lemma}\label{L25}
   Under Assumption \ref{as1-2}, the true solution of \eqref{SDDE} satisfies
  \begin{align}
       \sup \limits_{t \geq 0} \mathbb{E} |x(t)|^p \leq C_6 \Big(1+\sup \limits_{ -\tau \leq s \leq 0}\mathbb{E}(|x(s)|^p)\Big)
  \end{align}
  where $C_6$ is a positive constant.
\end{lemma}
\begin{proof}
    First, let $\lambda_1 > 0$ be the unique root of the following equation 
      \begin{align*}
           b_2- \frac{\lambda_1 2^{\frac{p}{2}+1}}{p} =b_3 e^{\lambda_1 \tau} .
      \end{align*}
      It is not difficult to see that for any $\lambda \in [0,\lambda_1 ]$
      \begin{align}\label{L25-1}
            b_2- \frac{\lambda 2^{\frac{p}{2}+1}}{p} -b_3 e^{\lambda \tau} \leq 0 .
      \end{align}
   By the It$\rm{\hat{o}}$ formula, we can derive that for any $t \geq 0$
    \begin{align*}
        e^{\lambda t}\Big(1+|x(t)|^2\Big)^{\frac{p}{2}} &= \Big(1+|x(0)|^2\Big)^{\frac{p}{2}} + \int_0^t \Big(\lambda e^{\lambda s} \big(1+|x(s)|^2\big)^{\frac{p}{2}}  + \frac{p}{2}e^{\lambda s}\big(1+|x(s)|^2\big)^{\frac{p}{2}-1} \\ \notag
        &\quad \big(2 \langle x(s), f(x(s), x(s-\tau)) \rangle + |g(x(s), x(s-\tau))|^2 + \frac{(p-2)|\langle x(s), g(x(s), x(s-\tau)) \rangle|^2}{1+|x(s)|^2}\big)\Big)ds \\ \notag
        &\quad + \int_0^t p e^{\lambda s}\big(1+|x(s)|^2\big)^{\frac{p}{2}-1} \big\langle x(s), g(x(s), x(s-\tau)) \big \rangle dW(s)
    \end{align*}
    This together with Assumption \ref{as1-2} implies
    \begin{align}\label{L25-2}
          e^{\lambda t}\mathbb{E}\Big(1+|x(t)|^2\Big)^{\frac{p}{2}} &\leq \mathbb{E}\Big(1+|x(0)|^2\Big)^{\frac{p}{2}} + \mathbb{E}\int_0^t \Big(\lambda e^{\lambda s} \big(1+|x(s)|^2\big)^{\frac{p}{2}}  + \frac{p}{2}e^{\lambda s}\big(1+|x(s)|^2\big)^{\frac{p}{2}-1} \\ \notag
        &\quad \big(2 \langle x(s), f(x(s), x(s-\tau)) \rangle + (p-1)|g(x(s), x(s-\tau))|^2\big)\Big)ds \\ \notag
        & \leq  \mathbb{E}\Big(1+|x(0)|^2\Big)^{\frac{p}{2}} + \mathbb{E}\int_0^t \Big(\lambda e^{\lambda s} \big(1+|x(s)|^2\big)^{\frac{p}{2}}  + \frac{p}{2}e^{\lambda s}\big(b_1 - b_2 |x(s)|^p + b_3|x(s-\tau)|^p \Big)ds \\ \notag
        &\leq \mathbb{E}\Big(1+|x(0)|^2\Big)^{\frac{p}{2}} + \mathbb{E}\int_0^t \Big(2^{\frac{p}{2}}\lambda+\frac{pb_1}{2}\Big)e^{\lambda s} ds \\ \notag
        & \quad+ \mathbb{E}\int_0^t \Big(\frac{p}{2}e^{\lambda s}\Big((\frac{\lambda 2^{\frac{p}{2}+1}}{p} - b_2 )|x(s)|^p + b_3|x(s-\tau)|^p \Big)\Big)ds 
    \end{align}
   Note that
    \begin{align}\label{L25-3}
        \mathbb{E}\int_0^t e^{\lambda s }|x(s-\tau)|^p ds &=   e^{\lambda\tau}\mathbb{E}\int_0^t e^{\lambda (s-\tau) }|x(s-\tau)|^p ds  \\ \notag 
        & = e^{\lambda\tau}\mathbb{E}\int_{-\tau}^0 e^{\lambda s}|x(s)|^p ds + e^{\lambda\tau}\mathbb{E}\int_0^t e^{\lambda s }|x(s)|^p ds -e^{\lambda\tau}\mathbb{E}\int_{t-\tau}^t e^{\lambda s }|x(s)|^p ds.
    \end{align}
 Hence, combining \eqref{L25-1}, \eqref{L25-2} and \eqref{L25-3}  together yields 
 \begin{align*}
          &\quad e^{\lambda t}\mathbb{E}\Big(1+|x(t)|^2\Big)^{\frac{p}{2}} \\ \notag
        &\leq \mathbb{E}\Big(1+|x(0)|^2\Big)^{\frac{p}{2}} +  \Big(2^{\frac{p}{2}}+\frac{pb_1}{2\lambda}\Big)\Big(e^{\lambda t} -1\Big) + \mathbb{E}\int_0^t \Big(\frac{p}{2}e^{\lambda s}(\frac{\lambda 2^{\frac{p}{2}+1}}{p} - b_2 +b_3 e^{\lambda \tau})|x(s)|^p \Big)ds \\ \notag
        &\quad  + \frac{pb_3}{2}e^{\lambda \tau} \Big(\mathbb{E}\int_{-\tau}^0 e^{\lambda s}|x(s)|^p ds -\mathbb{E}\int_{t-\tau}^t e^{\lambda s }|x(s)|^p ds\Big)  \\ \notag
         &\leq \mathbb{E}\Big(1+|x(0)|^2\Big)^{\frac{p}{2}} +  \Big(2^{\frac{p}{2}}+\frac{pb_1}{2\lambda}\Big)\Big(e^{\lambda t} -1\Big)  + \frac{pb_3}{2}e^{\lambda \tau} \Big(\mathbb{E}\int_{-\tau}^0 e^{\lambda s}|x(s)|^p ds -\mathbb{E}\int_{t-\tau}^t e^{\lambda s }|x(s)|^p ds\Big),
    \end{align*}
that is
\begin{align}\label{L25-4}
   &\quad \mathbb{E}\Big(1+|x(t)|^2\Big)^{\frac{p}{2}} \\ \notag
    &\leq e^{-\lambda t }\mathbb{E}\Big(1+|x(0)|^2\Big)^{\frac{p}{2}} +  \Big(2^{\frac{p}{2}}+\frac{pb_1}{2\lambda}\Big) + \frac{pb_3}{2}e^{\lambda \tau} \Big(\mathbb{E}\int_{-\tau}^0 e^{\lambda (s-t)}|x(s)|^p ds -\mathbb{E}\int_{t-\tau}^t e^{\lambda (s-t) }|x(s)|^p ds\Big) \\ \notag
    &\leq  e^{-\lambda t }\mathbb{E}\Big(1+|x(0)|^2\Big)^{\frac{p}{2}} +  \Big(2^{\frac{p}{2}}+\frac{pb_1}{2\lambda}\Big) + \frac{pb_3(e^{\lambda \tau}-1)e^{-\lambda t}}{2\lambda}\sup \limits_{ -\tau \leq s \leq 0}\mathbb{E}\Big(|x(s)|^p\Big).
\end{align}
Let 
\begin{align*}
    C_6:=2^{\frac{p}{2}}+ \Big(2^{\frac{p}{2}}+\frac{pb_1}{2\lambda}\Big) + \frac{pb_3(e^{\lambda \tau}-1)}{2\lambda}.
\end{align*}
The required assertion follows.
\end{proof}
In a way similar to the proof of Lemma \ref{L25}, we obtain the following result. 
\begin{lemma}\label{L27}
     Under Assumption \ref{as1-3}, for any two initial values $\xi,\eta \in L^2_{\mathcal{F}_0}([-\tau,0]; \RR^d)$, the corresponding two true solutions $x(t, \xi))$ and $x(t, \eta)$ satisfy
     \begin{align}
        \mathbb{E}\big(|x(t, \xi)-x(t, \eta)|^2\big) 
         &\leq \Big(1+\frac{b_5(e^{\gamma\tau}-1)}{\gamma}\Big)\Big(\sup \limits_{ -\tau \leq s \leq 0}\mathbb{E}\big(|x(s, \xi))-x(s, \eta)|^2\big)\Big)e^{-\gamma t},
     \end{align} 
     where $\gamma$ is a positive constant.
\end{lemma}

The following proof is similar to Lemma 3.7 in \cite{SHGL2022}, but for the sake of completeness, we present the process as follows. 
   \begin{lemma}\label{L28}
       Under Assumptions \ref{as1-1}, \ref{as1-2} and \ref{as1-3}, for any $T_2 > T_1 \geq 0$, recall the definition $z(t)=z(t;0,\xi)$ and $z_{T_1}:=\{z(T_1+\theta): -\tau\leq \theta \leq 0\}$, then the moment of the error between $x(t;T_1,z_{T_1})$ and $z(t;T_1,z_{T_1})$ satisfies
       \begin{align}
          \sup \limits_{T_1 \leq t \leq T_2} \mathbb{E}|x(t;T_1,z_{T_1}) -z(t;T_1,z_{T_1})|^2 \leq C_{7}^{T_2-T_1} \Big(1+ \sup \limits_{T_1-\tau \leq s \leq T_1}\mathbb{E}\big(|z(s)|^p\big)\Big) \Delta^{1-\nu},
       \end{align}
   \end{lemma}
   where $\nu$ is a positive constant and $C_{7}^{T_2-T_1}$ is a positive constant dependent on $T_2-T_1$. 
   \begin{proof}
       First, it not difficult to see that $z(t)$ has the following flow property as, e.g., in \cite{Mao2007} or \cite[p.78]{MY2006} 
       \begin{align*}
           z(t)=z(t;0,\xi)=z(t;s,z_s)\quad \quad \forall~ 0 \leq s < t < \infty,
\end{align*}
 provided $s$ is the multiple of $\Delta$. Then, define $\tau_{d}:= \inf \{t\geq T_1: |\overline{x}(t)| \geq d\}$, $\rho_{\Delta}:= \inf \{t\geq T_1: |\widetilde{y}(t)| > \Phi^{-1}(K\Delta^{-\nu})\}$, $\zeta_{\Delta}:= \inf \{t\geq T_1: |z(t)| > \Phi^{-1}(K\Delta^{-\nu})\}$ and $\chi_{\Delta}:= \tau_{\Phi^{-1}(K\Delta^{-\nu})} \wedge \rho_{\Delta} \wedge \zeta_{\Delta}$. For simplicity, let  $\overline{x}(t) :=x(t;T_1, z_{T_1})$ and $e(t):=\overline{x}(t) -z(t)$, by Young's inequality, we derive that for any $t \in [T_1, T_2]$ 
       \begin{align}\label{L28-1}
           \mathbb{E}|e(t)|^2 &=  \mathbb{E}\Big(|e(t)|^2 \boldsymbol{1}_{\{\chi_{\Delta}> T_2\} }\Big) +  \mathbb{E}\Big(|e(t)|^2 \boldsymbol{1}_{\{\chi_{\Delta} \leq  T_2\}}\Big) \\ \notag
           &\leq \mathbb{E}\Big(|e(t)|^2 \boldsymbol{1}_{\{\chi_{\Delta}> T_2\} }\Big) + \frac{2\Delta}{p}  \mathbb{E}|e(t)|^p+ \frac{p-2}{p \Delta^{\frac{2}{p-2}}} \mathbb{P}\{\chi_{\Delta} \leq  T_2\}.
       \end{align}
    Then we divide the proof into three steps.
    \begin{description}
\item[Step 1]  First, following a similar argument as in Lemma \ref{L25}, for any $t \in [T_1, T_2]$, we can easily get the following bounds.
\begin{align}\label{STEP1-1}
   &\quad \mathbb{E}\Big(1+|\overline{x}(t \wedge \tau_{d})|^2\Big)^{\frac{p}{2}} \\ \notag
    &\leq e^{-\lambda (t\wedge \tau_{d})}\mathbb{E}\Big(1+|\overline{x}(T_1)|^2\Big)^{\frac{p}{2}} +  \Big(2^{\frac{p}{2}}+\frac{pb_1}{2\lambda}\Big) \\ \notag
        &\quad+ \frac{pb_3}{2}e^{\lambda \tau} \Big(\mathbb{E}\int_{T_1-\tau}^{T_1} e^{\lambda (s-t\wedge \tau_{d})}|\overline{x}(s)|^p ds -\mathbb{E}\int_{t\wedge \tau_{d}-\tau}^{t\wedge \tau_{d}} e^{\lambda (s-t\wedge \tau_{d}) }|\overline{x}(s)|^p ds\Big) \\ \notag
        &\leq \mathbb{E}\Big(1+|\overline{x}(T_1)|^2\Big)^{\frac{p}{2}} +  \Big(2^{\frac{p}{2}}+\frac{pb_1}{2\lambda}\Big) + \frac{pb_3}{2}e^{\lambda \tau} \mathbb{E}\Big(\int_{T_1-\tau}^{T_1} |\overline{x}(s)|^p ds \Big) \\ \notag 
        &\leq \Big(2^{\frac{p}{2}+1}+\frac{pb_1}{2\lambda}\Big) + \Big(2^{\frac{p}{2}}+\frac{pb_3}{2}e^{\lambda \tau}  \Big) \Big(\sup \limits_{T_1-\tau \leq s \leq T_1}\mathbb{E} |z(s)|^p\Big).
        \end{align}
        Meanwhile, note that for any $d>0$
        \begin{align*}
            \mathbb{P}\{\tau_d \leq  T_2\}d^p \leq \mathbb{E}\Big(d^p \boldsymbol{1}_{\{\tau_d \leq  T_2\}}\Big)  \leq \mathbb{E}\Big(|\overline{x}(T_2 \wedge\tau_d)|^p \boldsymbol{1}_{\{\tau_d \leq  T_2\}}\Big) \leq \mathbb{E}\Big(1+|\overline{x}(T_2 \wedge \tau_{d})|^2\Big)^{\frac{p}{2}},
        \end{align*}
        that is
        \begin{align}\label{STEP1-2}
             \mathbb{P}\{\tau_d \leq  T_2\} \leq \frac{\mathbb{E}\Big(1+|\overline{x}(T_2 \wedge \tau_{d})|^2\Big)^{\frac{p}{2}}}{d^p}.
        \end{align}
        
 \item[Step 2] \label{step2}
      By employing the techniques used in the proofs of  Lemmas 3.2 in \cite{SHGL2022} and \eqref{STEP1-2}, we define two positive integers, $N_1$ and $N_2$, such that $N_1=T_1/\Delta$ and $N_2=T_2/\Delta$. Additionally, we also define $\beta_{\Delta} := \inf \{k \geq N_1 : |\widetilde{y}_k| \geq \Phi^{-1}(K\Delta^{-\nu})\}$. Then, for any $k \geq N_1$, it follows that
         \begin{align*}
           &\quad \mathbb{E}\Big( (1+|\widetilde{y}_{k\wedge\beta_{\Delta}}|^2)^{\frac{p}{2}} \Big| ~\mathcal{F}_{t_{(k-1)\wedge \beta_{\Delta}}}\Big)  \\ \notag 
           & \leq  \Big(1+|\widetilde{y}_{(k-1)\wedge\beta_{\Delta}}|^2\Big)^{\frac{p}{2}}+\Big(- \kappa \Delta\big(1+|\widetilde{y}_{(k-1)\wedge\beta_{\Delta}}|^2\big)^{\frac{p}{2}} + \frac{p\Delta}{2} \big(b_1  + 2\kappa - (b_2 - 2^{\frac{p}{2}+1}\kappa) |\overline{y}_{(k-1)\wedge\beta_{\Delta}}|^p \\ \notag
           &\quad+ b_3 |\overline{y}_{(k-1-M)\wedge\beta_{\Delta}}|^p \big)\Big)\mathbb{E}\Big( \boldsymbol{1}_{[N_1,~ \beta_{\Delta}]} (k)\Big| ~\mathcal{F}_{t_{(k-1)\wedge \beta_{\Delta}}}\Big),
       \end{align*}
       this implies
      \begin{align*}
             &\quad \mathbb{E}\Big (1+|\widetilde{y}_{k\wedge\beta_{\Delta}}|^2\Big)^{\frac{p}{2}}  \\ \notag 
             &\leq \mathbb{E}(1+|y_{N_1}|^2)^{\frac{p}{2}} + \frac{p(b_1+2\kappa)(k-N_1)\Delta}{2} \\ \notag
             &\quad-\frac{p(b_2 - 2^{\frac{p}{2}+1}\kappa)\Delta}{2}\sum \limits_{i=N_1}^{k-1}\mathbb{E}\Big[ |\overline{y}_{i\wedge\beta_{\Delta}}|^p\mathbb{E}\Big( \boldsymbol{1}_{[N_1,~ \beta_{\Delta}]} (i+1)\Big| ~\mathcal{F}_{t_{i\wedge \beta_{\Delta}}}\Big)\Big] \\ \notag
             &\quad + \frac{pb_3\Delta}{2}\sum \limits_{i=N_1}^{k-1-M}\mathbb{E}\Big[ |\overline{y}_{i\wedge\beta_{\Delta}}|^p\mathbb{E}\Big( \boldsymbol{1}_{[N_1,~ \beta_{\Delta}]} (i+1+M)\Big| ~\mathcal{F}_{t_{(i+M)\wedge \beta_{\Delta}}}\Big)\Big] \\ \notag
             &\quad + \frac{pb_3\Delta}{2}\sum \limits_{i=N_1-M}^{N_1-1}\mathbb{E}(|y_i|^p) .
      \end{align*}
      Due to the facts $\boldsymbol{1}_{[N_1,~ \beta_{\Delta}]} (i+1+M) \leq \boldsymbol{1}_{[N_1,~ \beta_{\Delta}]} (i+1)$, the inequality $(a+b)^p\leq 2^pa^p+2^pb^p$ and \eqref{L23-1}, we have
 \begin{align}\label{STEP2-1}
             &\quad \mathbb{E}\Big (1+|\widetilde{y}_{k\wedge\beta_{\Delta}}|^2\Big)^{\frac{p}{2}}   \\ \notag 
             &\leq \mathbb{E}(1+|y_{N_1}|^2)^{\frac{p}{2}} + \frac{p(b_1+2\kappa)(k-N_1)\Delta}{2}  + \frac{pb_3\tau}{2}\sup \limits_{ N_1-M \leq i \leq N_1-1}\mathbb{E}\Big(|y_{i}|^p\Big) \\ \notag
             &\leq \Big(2^{\frac{p}{2}}+\frac{p(b_1+2\kappa)(k-N_1)\Delta}{2}\Big) + \Big(2^{\frac{p}{2}} +\frac{pb_3\tau}{2}\Big)\sup \limits_{ T_1-\tau \leq s \leq T_1}\mathbb{E}\Big(|z(s)|^p\Big).
      \end{align}
       Meanwhile, note that
       \begin{align*}
          \Big(  \Phi^{-1}(K\Delta^{-\nu}) \Big)^p \mathbb{P}\{\rho_d \leq  T_2\} \leq \mathbb{E}\Big(|\widetilde{y}(T_2 \wedge \rho_{\Delta})|^p\boldsymbol{1}_{\{\rho_{\Delta} \leq  T_2\}}\Big)\leq \mathbb{E}\Big (1+|\widetilde{y}_{_{N_2\wedge\beta_{\Delta}}}|^2\Big)^{\frac{p}{2}} ,
       \end{align*}
       that is
        \begin{align}\label{STEP2-2}
            \mathbb{P}\{\rho_d \leq  T_2\} \leq \frac{\mathbb{E}\Big (1+|\widetilde{y}_{_{N_2\wedge\beta_{\Delta}}}|^2\Big)^{\frac{p}{2}} }{\Big(\Phi^{-1}(K\Delta^{-\nu})\Big)^p}.
        \end{align}
        \item[Step 3] \label{step3}
        As for $z(t)$, by \eqref{CDTSDDE2} and \eqref{CDTSDDE}, let 
            \begin{align*}
               F(t) = f(y(t), y(t-\tau)),\quad G(t) = g(y(t), y(t-\tau)),
            \end{align*}
        then we can see that for any $t \in [T_1, \zeta_{\Delta}]$
        \begin{align*}
            z(t)= z(T_1) + \int_{T_1}^t F(s)ds + \int_{T_1}^t G(s)dW(s).
        \end{align*}
Applying the  It$\rm{\hat{o}}$ formula to the above equation yields
        \begin{align}\label{STEP3-37}
            &\quad e^{\alpha (t \wedge \zeta_{\Delta})} \mathbb{E}\Big(1+|z(t \wedge \zeta_{\Delta})|^2\Big)^{\frac{p}{2}} \\ \notag
       &\leq \mathbb{E}\Big(1+|z(T_1)|^2\Big)^{\frac{p}{2}} + \mathbb{E}\int_{T_1}^{t \wedge \zeta_{\Delta}} \Big(\alpha e^{\alpha s} \big(1+|z(s)|^2\big)^{\frac{p}{2}}  + \frac{p}{2}e^{\alpha s}\big(1+|z(s)|^2\big)^{\frac{p}{2}-1} \\ \notag
       &\quad \times\big(2 \langle z(s), F(s) \rangle + (p-1)|G(s)|^2\big)\Big)ds \\ \notag 
        & \leq  \mathbb{E}\Big(1+|z(T_1)|^2\Big)^{\frac{p}{2}} + \mathbb{E}\int_{T_1}^{t \wedge \zeta_{\Delta}} \Big(\alpha e^{\alpha s} \big(1+|z(s)|^2\big)^{\frac{p}{2}}  + \frac{p}{2}e^{\alpha s}\big(1+|z(s)|^2\big)^{\frac{p}{2}-1} \\ \notag
      & \quad \times\big(2\langle z(s)-y(s), F(s) \rangle+ 2\langle y(s), F(s) \rangle + (p-1)|G(s)|^2\big) \Big)ds \\ \notag
        & \leq \mathbb{E}\Big(1+|z(T_1)|^2\Big)^{\frac{p}{2}} + \mathbb{E}\int_{T_1}^{t \wedge \zeta_{\Delta}} \Big(\alpha e^{\alpha s} \big(1+|z(s)|^2\big)^{\frac{p}{2}}  + \frac{p}{2}e^{\alpha s}\big(1+|z(s)|^2\big)^{\frac{p}{2}-1} \\ \notag 
       & \quad \times \big(2|z(s)-y(s)||F(s)|+ \overline{b_1}+ \overline{b_3}|y(s-\tau)|^2\big) \Big)ds \\ \notag
       & \leq \mathbb{E}\Big(1+|z(T_1)|^2\Big)^{\frac{p}{2}} + \mathbb{E}\int_{T_1}^{t \wedge \zeta_{\Delta}} e^{\alpha s} \Big( \big(\alpha + \frac{p-2}{2}\big)\big(1+|z(s)|^2\big)^{\frac{p}{2}} + \big(2^{\frac{3p}{2}}|z(s)-y(s)|^{\frac{p}{2}} |F(s)|^{\frac{p}{2}}\big) \\ \notag
       & \quad + \big(2^p\overline{b}_1^{\frac{p}{2}}+2^p\overline{b}_3^{\frac{p}{2}}|y(s-\tau)|^p\big)\Big)ds \\ \notag
       & := \mathbb{E}\Big(1+|z(T_1)|^2\Big)^{\frac{p}{2}} + \int_{T_1}^{t \wedge \zeta_{\Delta}} e^{\alpha s} \Big(J_1+J_2+J_3 \Big)ds
        \end{align}
       Subsequently we will estimate $J_1$, $J_2$ and $J_3$ separately. Similar to the proof of \eqref{L23-6}, for any $k \in [N_1, N_2]$, we have
       \begin{align*}
           \mathbb{E}\Big(|y_k|^p\Big) &\leq \mathbb{E}\Big(1+|y_k|^2\Big)^{\frac{p}{2}} \\ \notag
           &\leq \mathbb{E}\Big(1+|y_{{N_1}}|^2\Big)^{\frac{p}{2}}+ \frac{p}{2\kappa}\Big(b_1+2^{\frac{p}{2}+1}\kappa\Big)+ \frac{pb_3e^{2\kappa \tau}}{2\kappa}\Big( \sup \limits_{N_1-M \leq i \leq N_1}\mathbb{E}\big(|y_i|^p\big)\Big)
       \end{align*} 
       This implies that for any $k \in [N_1-M, N_2]$
       \begin{align}\label{STEP3-38}
           \mathbb{E}\Big(|y_k|^p\Big) &\leq \mathbb{E}\Big(1+|y_k|^2\Big)^{\frac{p}{2}} \\ \notag
           & \leq 2^{\frac{p}{2}} + \frac{p}{2\kappa}\Big(b_1+2^{\frac{p}{2}+1}\kappa\Big) + \Big(2^{\frac{p}{2}}+ \frac{pb_3e^{2\kappa \tau}}{2\kappa}\Big)\Big(\sup \limits_{N_1-M \leq i \leq N_1} \mathbb{E}(|y_i|^p)\Big)
       \end{align}
       Then
       \begin{align}\label{STEP3-39}
           J_1&: =(\alpha + \frac{p-2}{2})\mathbb{E}\Big(1+|z(s)|^2\Big)^{\frac{p}{2}}  \\ \notag
           & \leq (\alpha + \frac{p-2}{2}) \mathbb{E}\Big(1+2|z(s)-y(s)|^2+2|y(s)|^2\Big)^{\frac{p}{2}} \\ \notag
           & \leq 2^p (\alpha+ \frac{p-2}{2})\Big(1+\mathbb{E}(|z(s)-y(s)|^p)+ \mathbb{E}(|y(s)|^p)\Big)
       \end{align}
       Note that from \eqref{L23-7}, we know that for any $s \in [t_k,t_{k+1}]$
        \begin{align}\label{STEP3-40}
           & \quad \mathbb{E}\Big(|z(s) - y(s)|^p\Big) \leq C_5 (1+\mathbb{E}|y_k|^p) \Delta^{\frac{p(1-\upsilon)}{2}},
       \end{align}
       Thus
       \begin{align}\label{STEP3-41}
           J_1 \leq 2^p\big(\alpha + \frac{p-2}{2}\big)\big((C_5+1)(1+\mathbb{E}(|y_k|^p))\big)
       \end{align}
       Next, by employing \eqref{CON}, \eqref{L23-7}, \eqref{STEP3-41}, the H$\rm{\ddot{o}}$lder and Young inequalities, and recalling that $\upsilon \in (0, \frac{1}{3}]$, one can derive that for any $s \in [t_k, t_{k+1}]$
       \begin{align}\label{STEP3-42}
           J_2 &=2^{\frac{3p}{2}}\mathbb{E}\Big(|z(s)-y(s)|^{\frac{p}{2}}|F(s)|^{\frac{p}{2}}\Big) \\ \notag
         & \leq 2^{\frac{3p}{2}}\Big(\mathbb{E}(|z(s)-y(s)|^p)\Big)^{\frac{1}{2}}\Big(\mathbb{E}(|F(s)|^p)\Big)^{\frac{1}{2}} \\ \notag
         & \leq 2^{\frac{3p}{2}}K^{\frac{p}{2}}\Big(C_5\big(1+\mathbb{E}|y_k|^p\big)\Delta^{\frac{p(1-\upsilon)}{2}}\Big)^{\frac{1}{2}} \Big(\mathbb{E}\big(1+|y(s)|\big)^p \Delta^{-p \upsilon}\Big)^{\frac{1}{2}} \\ \notag
         & \leq 2^{\frac{3p}{2}-1}K^{\frac{p}{2}}\Big(C_5(1+\mathbb{E}|y_k|^p)+ \mathbb{E}\big(1+|y(s)|\big)^p\Big) \\ \notag
         & \leq 2^{\frac{3p}{2}-1}K^{\frac{p}{2}}\Big(C_5+2^p+ (C_5+2^p)\mathbb{E}|y_k|^p\Big)
       \end{align}
      As for $J_3$, it is obvious that
      \begin{align}\label{STEP3-43}
          J_3 = 2^p \overline{b}_1^{\frac{p}{2}}+ 2^p \overline{b}_3^{\frac{p}{2}}\mathbb{E}(|y(s-\tau)|^p)
      \end{align}
      Combining \eqref{STEP3-37}, \eqref{STEP3-38}, \eqref{STEP3-40}, \eqref{STEP3-41}, \eqref{STEP3-42} and \eqref{STEP3-43} together, it is not difficult to derive that
       \begin{align}\label{STEP3-7}
        \mathbb{E}\Big(1+|z(t \wedge \zeta_{\Delta})|^2\Big)^{\frac{p}{2}}&\leq C_8^{T_2-T_1}\Big(1+ \sup \limits_{ T_1-\tau \leq s \leq T_1}\mathbb{E}\big(|y(s)|^p\big)\Big) \\ \notag
        &\leq C_8^{T_2-T_1}\Big(1+ \sup \limits_{ T_1-\tau \leq s \leq T_1}\mathbb{E}\big(|z(s)|^p\big)\Big),
     \end{align}
     where $C_8^{T_2-T_1}$ is a positive constant dependent on $p, \alpha, C_5, \kappa, K, b_1, b_3,\overline{b}_1,\overline{b}_3, \tau, T_2-T_1$. 
 Meanwhile, note that
       \begin{align*}
          \Big(  \Phi^{-1}(K\Delta^{-\nu}) \Big)^p \mathbb{P}\{\zeta_{\Delta} \leq  T_2\} \leq \mathbb{E}\Big(|z(T_2 \wedge \zeta_{\Delta})|^p\boldsymbol{1}_{\{\zeta_{\Delta} \leq  T_2\}}\Big)\leq \mathbb{E}\Big (1+|z(T_2 \wedge \zeta_{\Delta})|^2\Big)^{\frac{p}{2}} ,
       \end{align*}
       that is
        \begin{align}\label{STEP3-8}
            \mathbb{P}\{\zeta_{\Delta} \leq T_2\} \leq \frac{\mathbb{E}\Big (1+|z(T_2 \wedge \zeta_{\Delta})|^2\Big)^{\frac{p}{2}}}{\Big(\Phi^{-1}(K\Delta^{-\nu})\Big)^p}.
        \end{align}
    \item[Step 4] In a way similar to the proof of Lemmas \ref{L23} and \ref{L25}, it is easy to see that for any $t \in [T_1, T_2]$
    \begin{align}\label{STEP4-1}
        \frac{2\Delta}{p}  \mathbb{E}|e(t)|^p \leq \frac{2^{p+1}\Delta}{p} \big(\mathbb{E}|\overline{x}(t)|^p + \mathbb{E}|z(t)|^p \big) \leq C_9 \Big(1+ \sup \limits_{ T_1-\tau \leq s \leq T_1}\mathbb{E}\big(|z(s)|^p\big)\Big) \Delta,
    \end{align}
    where $C_9$ is a positive constant dependent on $C_3, C_6$ and $p$. Then by \eqref{STEP1-1}, \eqref{STEP1-2}, \eqref{STEP2-1}, \eqref{STEP2-2}, \eqref{STEP3-7}, \eqref{STEP3-8}, \eqref{PHI-1} and the previous setting $\nu :=\frac{q+2}{p-2}$, we get
    \begin{align}\label{STEP4-2}
        \frac{p-2}{p \Delta^{\frac{2}{p-2}}} \mathbb{P}\{\chi_{\Delta} &\leq  T_2\}\leq   \frac{p-2}{p \Delta^{\frac{2}{p-2}}} \Big(\mathbb{P}\{\tau_{\Phi^{-1}(K\Delta^{-\nu})}\leq  T_2\}+ \mathbb{P}\{\rho_{\Delta} \leq  T_2\}+ \mathbb{P}\{\zeta_{\Delta} \leq  T_2\}\Big) \\ \notag 
        &\leq  \frac{p-2}{p \Delta^{\frac{2}{p-2}}} \frac{C_{10}^{T_2-T_1}}{\big(\Phi^{-1}(K\Delta^{-\nu})\big)^p}  \Big(1+ \sup \limits_{ T_1-\tau \leq s \leq T_1}\mathbb{E}\big(|z(s)|^p\big)\Big) \\ \notag
        &\leq C_{11}^{T_2-T_1} \Big(1+ \sup \limits_{ T_1-\tau \leq s \leq  T_1}\mathbb{E}\big(|z(s)|^p\big)\Big)\Delta^{\frac{p}{p-2}-\frac{2}{p-2}} \\ \notag
        &= C_{11}^{T_2-T_1} \Big(1+ \sup \limits_{ T_1-\tau \leq s \leq  T_1}\mathbb{E}\big(|z(s)|^p\big)\Big)\Delta,
    \end{align}
    where $C_{10}^{T_2-T_1}$ and $C_{11}^{T_2-T_1}$ are a pair of constants dependent on $p, b_1, b_3, \kappa, \lambda, \tau, C_8^{T_2-T_1} $. Next, employing the It$\rm{\hat{o}}$ formula and Assumption \ref{as1-3}, we get that for any $t \in [T_1, T_2]$
    \begin{align}\label{STEP4-3}
       &\quad \Big|e(t \wedge \chi_{\Delta})\Big|^2 \\ \notag
        &\leq \int_{T_1}^{t \wedge \chi_{\Delta}}\Big(2\langle e(s), f(\overline{x}(s), \overline{x}(s-\tau)) - f(y(s), y(s-\tau)) \rangle + |g(\overline{x}(s), \overline{x}(s-\tau))-g(y(s), y(s-\tau))|^2\Big)ds \\ \notag
        &\quad +  \int_{T_1}^{t \wedge \chi_{\Delta}} \Big(2\langle e(s), g(\overline{x}(s), \overline{x}(s-\tau)) - g(y(s), y(s-\tau)) \rangle \Big)dW(s). 
    \end{align}
   Let $\overline{\kappa} := \sigma -1$, then by the Young inequality and \eqref{S3-1-1}, we have 
   \begin{align}\label{STEP4-4}
     &\quad  2\langle e(s), f(\overline{x}(s), \overline{x}(s-\tau)) - f(y(s), y(s-\tau)) \rangle + |g(\overline{x}(s), \overline{x}(s-\tau))-g(y(s), y(s-\tau))|^2 \\ \notag
     &\leq 2\langle e(s), f(\overline{x}(s), \overline{x}(s-\tau)) - f(z(s), z(s-\tau)) \rangle + 2\langle e(s), f(z(s), z(s-\tau)) - f(y(s), y(s-\tau)) \rangle \\ \notag
     &\quad + (1+ \overline{\kappa}) |g(\overline{x}(s), \overline{x}(s-\tau))-g(z(s), z(s-\tau))|^2 + (1+\frac{1}{\overline{\kappa}}) |g(z(s), z(s-\tau))-g(y(s), y(s-\tau))|^2 \\ \notag
     &\leq -b_4 |e(s)|^2 + b_5|e(s- \tau)|^2 + 2 a_1 |e(s)|\Big(|z(s)-y(s)| + |z(s-\tau)-y(s-\tau)|\Big) \Big(1+ |z(s)|^q \\ \notag
     &\quad + |z(s-\tau)|^q +|y(s)|^q+ |y(s-\tau)|^q\Big)+ a_1 (1+\frac{1}{\overline{\kappa}}) \Big(|z(s)-y(s)|^2 + |z(s-\tau)-y(s-\tau)|^2\Big) \\ \notag
     &\quad\times\Big(1 + |z(s)|^q + |z(s-\tau)|^q +|y(s)|^q+ |y(s-\tau)|^q\Big).
   \end{align}
   Inserting \eqref{STEP4-4} into \eqref{STEP4-3} and using the fact $z(t)=\overline{x}(t)$ for any $t \in [T_1-\tau, T_1]$, we can derive that
   \begin{align}\label{STEP4-5}
        &\quad \mathbb{E}\Big(\Big|e(t \wedge \chi_{\Delta})\Big|^2 \Big)\\ \notag
        &\leq \mathbb{E}\int_{T_1}^{t \wedge \chi_{\Delta}} 2a_1|e(s)|^2 + a_1|z(s)-y(s)|^2(1 + |z(s)|^q + |z(s-\tau)|^q +|y(s)|^q+ |y(s-\tau)|^q)^2 \\ \notag
        &\quad+ a_1|z(s-\tau)-y(s-\tau)|^2(1 + |z(s)|^q + |z(s-\tau)|^q +|y(s)|^q+ |y(s-\tau)|^q)^2 \\ \notag
        &\quad + a_1 (1+\frac{1}{\overline{\kappa}}) |z(s)-y(s)|^2  \Big(1 + |z(s)|^q + |z(s-\tau)|^q +|y(s)|^q+ |y(s-\tau)|^q\Big) \\ \notag
        &\quad +  a_1 (1+\frac{1}{\overline{\kappa}}) |z(s-\tau)-y(s-\tau)|^2  \Big(1 + |z(s)|^q + |z(s-\tau)|^q +|y(s)|^q+ |y(s-\tau)|^q\Big) ds\\ \notag
        &\leq 2a_1  \mathbb{E}\int_{T_1}^{t \wedge \chi_{\Delta}} |e(s)|^2 ds + a_1 (2+\frac{1}{\overline{\kappa}}) \mathbb{E}\int_{T_1}^{t \wedge \chi_{\Delta}} |z(s)-y(s)|^2(1 + |z(s)|^q + |z(s-\tau)|^q +|y(s)|^q \\ \notag
        &\quad+ |y(s-\tau)|^q)^2 ds + a_1 (2+\frac{1}{\overline{\kappa}}) \mathbb{E}\int_{T_1}^{t \wedge \chi_{\Delta}} |z(s-\tau)-y(s-\tau)|^2(1 + |z(s)|^q + |z(s-\tau)|^q +|y(s)|^q \\ \notag
        &\quad+ |y(s-\tau)|^q)^2 ds  \\ \notag
        &:= 2a_1  \mathbb{E}\int_{T_1}^{t \wedge \chi_{\Delta}} |e(s)|^2 ds  +J_1+J_2.
   \end{align}
   Using the H$\rm{\ddot{o}}$lder inequality and the Young inequality, it follows from \eqref{L23-7}, \eqref{L23-8} and \eqref{STEP3-38} that
   \begin{align}\label{STEP4-6}
       J_1  
       &\leq  a_1 (2+\frac{1}{\overline{\kappa}}) \mathbb{E}\int_{T}^{T_2} |z(s)-y(s)|^2(1 + |z(s)|^q + |z(s-\tau)|^q +|y(s)|^q + |y(s-\tau)|^q)^2 ds \\ \notag 
       &\leq  25a_1 (2+\frac{1}{\overline{\kappa}}) \int_{T_1}^{T_2} \Big(\mathbb{E}|z(s)-y(s)|^4\Big)^{\frac{1}{2}}\Big(\mathbb{E}(1 + |z(s)|^{4q} + |z(s-\tau)|^{4q} +|y(s)|^{4q} + |y(s-\tau)|^{4q}) \Big)^{\frac{1}{2}}ds\\ \notag 
       &\leq 25a_1 (2+\frac{1}{\overline{\kappa}}) \int_{T_1}^{T_2} \Big(\mathbb{E}|z(s)-y(s)|^p\Big)^{\frac{2}{p}}\Big(\mathbb{E}(1 + |z(s)|^{4q} + |z(s-\tau)|^{4q} +|y(s)|^{4q} + |y(s-\tau)|^{4q}) \Big)^{\frac{1}{2}}ds\\ \notag 
       &\leq 25a_1C_5^{\frac{2}{p}} (2+\frac{1}{\overline{\kappa}})\Delta^{1-\nu} \int_{T_1}^{T_2} \Big(1+\mathbb{E}|y(s)|^p\Big)^{\frac{2}{p}}\Big(\mathbb{E}(1 + |z(s)|^{4q} + |z(s-\tau)|^{4q} +|y(s)|^{4q}  \\ \notag
       &\quad+ |y(s-\tau)|^{4q}) \Big)^{\frac{1}{2}}ds\\ \notag 
     &\leq \frac{25}{2} a_1C_5^{\frac{2}{p}} (2+\frac{1}{\overline{\kappa}})\Delta^{1-\nu} \int_{T_1}^{T_2} \Big((1+\mathbb{E}|y(s)|^p)^{\frac{4}{p}}+ 1 + \mathbb{E}|z(s)|^{4q} + \mathbb{E}|z(s-\tau)|^{4q} +\mathbb{E}|y(s)|^{4q}  \\ \notag
       &\quad+ \mathbb{E}|y(s-\tau)|^{4q} \Big)ds\\ \notag 
     & \leq \frac{25}{2} a_1C_5^{\frac{2}{p}} (2+\frac{1}{\overline{\kappa}})\Delta^{1-\nu}  \int_{T_1}^{T_2} \Big((2+  \frac{4(p-4q)}{p}) + \frac{4}{p} \mathbb{E}|y(s)|^p+ \frac{4q}{p} (\mathbb{E}|z(s)|^{p} + \mathbb{E}|z(s-\tau)|^{p} \\ \notag
     &\quad+\mathbb{E}|y(s)|^{p} + \mathbb{E}|y(s-\tau)|^{p} )\Big)ds \\ \notag
     &\leq C_{12}^{T_2-T_1} \Big(1+ \sup \limits_{ T_1-\tau \leq s \leq T_1}\mathbb{E}\big(|z(s)|^p\big)\Big)\Delta^{1-\nu},
   \end{align}
   where $C_{12}^{T_2-T_1}$ is a positive constant dependent on $a_1, p, q, \overline{\kappa},\tau, T, C_3$ and $C_5$. Similarly, by assumption \ref{as1-4},  we can also get
   \begin{align}\label{STEP4-7}
     J_2 
      &\leq a_1 (2+\frac{1}{\overline{\kappa}}) \mathbb{E}\int_{T_1}^{T_2} |z(s-\tau)-y(s-\tau)|^2(1 + |z(s)|^q + |z(s-\tau)|^q +|y(s)|^q \\ \notag
        &\quad+ |y(s-\tau)|^q)^2 ds  \\ \notag
       &\leq 25a_1 (2+\frac{1}{\overline{\kappa}}) \int_{T_1}^{T_2} \Big(\mathbb{E}|z(s-\tau)-y(s-\tau)|^p\Big)^{\frac{2}{p}}\Big(\mathbb{E}(1 + |z(s)|^{4q} + |z(s-\tau)|^{4q} +|y(s)|^{4q} \\ \notag
       &\quad + |y(s-\tau)|^{4q}) \Big)^{\frac{1}{2}}ds\\ \notag 
       &\leq 25a_1C_5^{\frac{2}{p}} (2+\frac{1}{\overline{\kappa}})\Delta^{1-\nu} \int_{T_1}^{T_2} \Big(1+\mathbb{E}|y(s-\tau)|^p\Big)^{\frac{2}{p}}\Big(\mathbb{E}(1 + |z(s)|^{4q} + |z(s-\tau)|^{4q} +|y(s)|^{4q}  \\ \notag
       &\quad+ |y(s-\tau)|^{4q}) \Big)^{\frac{1}{2}}ds  \\ \notag
       &\leq C_{13}^{T_2-T_1} \Big(1+ \sup \limits_{ T_1-\tau \leq s \leq T_1}\mathbb{E}\big(|z(s)|^p\big)\Big)\Delta^{1-\nu},
   \end{align}
    where $C_{13}^{T_2-T_1}$ is a positive constant dependent on $a_1, p, q, \overline{\kappa},\tau, T_2-T_1, K_1, C_3$ and $C_5$. Inserting \eqref{STEP4-6} and \eqref{STEP4-7} into \eqref{STEP4-5} and using the Gronwall inequality yields
    \begin{align}\label{STEP4-8}
         \mathbb{E}\Big(\Big|e(t \wedge \chi_{\Delta})\Big|^2 \Big)\leq \Big(C_{12}^{T_2-T_1}+C_{13}^{T_2-T_1}\Big)e^{4a_1(T_2-T_1)} \Big(1+ \sup \limits_{ T_1-\tau \leq s \leq T_1}\mathbb{E}\big(|z(s)|^p\big)\Big)\Delta^{1-\nu}
    \end{align}
    Meanwhile, note that for any $t \in [T_1, T_2]$
    \begin{align}\label{STEP4-9}
        \mathbb{E}\Big(|e(t)|^2 \boldsymbol{1}_{\{\chi_{\Delta}> T_2\} }\Big) \leq  \mathbb{E}\Big(\Big|e(t \wedge \chi_{\Delta})\Big|^2 \Big)
    \end{align}
    Thus, inserting \eqref{STEP4-1}, \eqref{STEP4-2}, \eqref{STEP4-8} and \eqref{STEP4-9} into \eqref{L28-1} gives
 \begin{align}\label{STEP4-10}
           \mathbb{E}|e(t)|^2 
           &\leq \mathbb{E}\Big(|e(t)|^2 \boldsymbol{1}_{\{\chi_{\Delta}> T_2\} }\Big) + \frac{2\Delta}{p}  \mathbb{E}|e(t)|^q+ \frac{p-2}{p \Delta^{\frac{2}{p-2}}} \mathbb{P}\{\chi_{\Delta} \leq  T_2\} \\ \notag
           &\leq \Big(C_{12}^{T_2-T_1}+C_{13}^{T_2-T_1}\Big)e^{4a_1(T_2-T_1)} \Big(1+ \sup \limits_{ T-\tau \leq s \leq T}\mathbb{E}\big(|z(s)|^p\big)\Big)\Delta^{1-\nu} \\ \notag
           &\quad+ \Big(C_9+ C_{11}^{T_2-T_1}\Big) \Big(1+ \sup \limits_{ T_1-\tau \leq s \leq T_1}\mathbb{E}\big(|z(s)|^p\big)\Big) \Delta \\ \notag
           &\leq C_2^{T_2-T_1} \Big(1+ \sup \limits_{ T-\tau \leq s \leq T}\mathbb{E}\big(|z(s)|^p\big)\Big) \Delta^{1-\nu},
       \end{align}
       where $C_7^{T_2-T_1} :=\Big((C_{12}^{T_2-T_1}+C_{13}^{T_2-T_1})e^{2a_1(T_2-T_1)} + C_9+ C_{11}^{T_2-T_1}\Big)\Big(1+C_{\tau}\Big)$. 
       The proof is complete. 
\end{description}
   \end{proof}
   From Lemmas \ref{L23}, \ref{L27}, and \ref{L28}, we can see that Condition \ref{con1} hold, then by Theorem \ref{T--2}, we can conclude this part by the following theorems.
\begin{theorem}
    Under Assumption \ref{as1-1}-\ref{as1-4}, the numerical solution $z(t;0 \xi)$ converges strongly to the true solution $x(t,;0\xi)$ in the infinite horizon, i.e.,
    \begin{align*}
          \sup \limits_{ t \geq 0} \mathbb{E}|x(t;0,\xi) -z(t;0,\xi)|^2 \leq C \Delta^{1-\nu},
    \end{align*}
    where $C$ is a positive constant independent of $t$.
\end{theorem}

\section{Backward Euler-Maruyama method}
In the previous section, to show the applicability of Theorem \ref{T--2}, we used an explicit method as an example. Similarly, Theorem \ref{T-22} also requires an example to illustrate its applicability. It is worth noting that when solving non-stiff SDEs, the explicit method performs very well. However, when dealing with stiff SDEs, the implicit method usually faces severe step size limits due to stability considerations, which will cause the overall computational cost to become expensive. In this case, the implicit method with better stability properties will be a better choice, see \cite{Hu1996, MPS1998,RVB2010} for more details. Thus, in this section, to show the generality of our technique, we will consider the Backward Euler-Maruyama method under the following assumptions.

\begin{assumption}\label{as1}
For any  $x,~\bar{x},~y,~\bar y\in \RR ^d $ , there exists a positive
constant $a_1 $ such that
\begin{align*}
  | f(x, y)-f(\bar{x},\bar  y)|^2\vee | g(x, y)-g(\bar{x}, \bar y)|^2 \leq a_1 (| x-\bar {x}|^2+| y-\bar {y}|^2).
 \end{align*}
\end{assumption}
\begin{assumption}\label{as2}
   There exist nonnegative constants  $b_1,~b_2,~b_3, ~b_4$ and $b_5$ with $b_1>b_2,~b_3>b_4$,  such that
\begin{align*}
2\big\langle x-\bar {x} ,f(x,y)-f(\bar {x},\bar y)\big\rangle&+| g(x,y)-g(\bar {x},\bar{y})| ^2
\leq -b_1| x-\bar {x}|^2+b_2| y-\bar {y}| ^2 \\
2\big \langle x, f(x,y) \big \rangle & + |g(x,y)|^2 \leq -b_3|x|^2+b_4|y|^2 + b_5
\end{align*}

for any $x,~\bar x,~y,~\bar y\in\RR^d$.
\end{assumption}
\begin{assumption} \label{as3}
     The initial value is Holder continuous, i.e., there exist a positive constant $K_1$ such that for any $-\tau \leq s < t \leq 0$
     \begin{align}
         \mathbb{E}|\xi(t)-\xi(s)|^2 \leq K_1(t-s).
     \end{align}
\end{assumption}
From Assumption \ref{as1}, for any $x,~y \in\RR^d $, one can see that
\begin{equation}\label{LG}
    |f(x,y)|^2 \vee |g(x,y)|^2 \leq c_1 (|x|^2+|y|^2) +c_2,
\end{equation}
where $x_1 = 2a_1, c_2=2(|f(0,0)|^2 \vee |g(0,0)|^2)$.

It should be pointed out that under Assumption \ref{as1} and \ref{as2},  SDDE \eqref{SDDE} with the initial data $\xi\in\mathcal C_{\mathcal F_0}^{p}$ has a unique global solution $x(t)$ for $t\geq -\tau$  (see \cite[Theorem 2.4]{MR2005} or \cite[p.278, Theorem 7.12]{MY2006} for details).

Now, we define the backward Euler-Maruyama approximate solution for Eq. \ref{SDDE} as follows.
\begin{align}\label{DSDDE}
    \begin{split}
  \left \{
 \begin{array}{ll}
X_k=\xi(t_k), \quad\quad\quad\quad k=-M,-M+1,...,0,\\
 X_k=X_{k-1}+f(X_k, X_{k-M})\Delta+ g(X_{k-1}, X_{k-M-1})\Delta W_{k-1}, \quad \quad k=1,2,...,
 \end{array}
 \right.
 \end{split}
 \end{align}
 where $\Delta W_{k-1} := W(k\Delta)-W((k-1)\Delta)$.
And the continuous backward Euler-Maruyama approximate solution is defined by
\begin{align}\label{CDSDDE}
    \begin{split}
  \left \{
 \begin{array}{ll}
\overline{X}(t)=\xi(t), \quad\quad\quad\quad -\tau \leq t \leq 0,\\
\overline{X}(t)=\xi(0)+\int_{0}^{t}f(X_1(s), X_{1}(s-\tau)) ds+ \int_{0}^{t}g(X_2(s), X_{2}(s-\tau))dW(s), \quad  t\geq 0,
 \end{array}
 \right.
 \end{split}
 \end{align}
 where $X_1(t) := \sum \limits_{k=0}^{\infty} X_{k+1} \boldsymbol{1}_{[kh, (k+1)h)}(t)$ and $X_2(t) := \sum \limits_{k=0}^{\infty} X_{k} \boldsymbol{1}_{[kh, (k+1)h)}(t)$ for any $k \geq -M$. Similar to the true solution of Eq.\eqref{SDDE}, denote $\overline{X}_{t}:=\{\overline{X}(t+\theta) : -\tau \leq \theta \leq 0\}$, and sometimes to emphasize the initial value $\xi$ at $t=0$, we also denote them by $X_{k}^{0, \xi}$ and $\overline{X}_{t}^{0, \xi}$.

 \begin{assumption}\label{as4}
     Suppose that for any fixed $\Delta \in (0, \Delta^*] \subset (0,1]$, there exist two functions $G_{\Delta}: \mathbb{R}^d \to \mathbb{R}^d$ and $F_{\Delta}: \mathbb{R}^d\to \mathbb{R}^d$, such that for any $x,y \in \mathbb{R}^d$, $x-f(x,y)\Delta$ can be expressed as $G_{\Delta}(x)-F_{\Delta}(y )$, and $G_{\Delta}$ has its inverse function $G_{\Delta}^{-1}: \mathbb{R}^d \to \mathbb{R}^d$.
 \end{assumption}
 Thus, under Assumption \ref{as4}, for any $k > 0$, \eqref{DSDDE} can be represented as
 \begin{align}\label{GI}
     X_k = G_{\Delta}^{-1}\Big(F_{\Delta}(X_{k-M}) + X_{k-1}+ g(X_{k-1}, X_{k-M-1})\Delta W_{k-1}\Big).
 \end{align}
Next, let $\mathcal{F}_t:= \sigma (W(u), 0 \leq u \leq t) \wedge \mathcal{N}$ for $t > 0$, where $\mathcal{N}$ is the set of all $\mathbb{P}$-null sets. With this definition, we can proceed to the following lemma to show that the first condition in Condition \ref{con19} is satisfied. 
 \begin{lemma}\label{L--5}
 Under Assumption \ref{as4}, $\{\overline{X}_{t_k}\}_{k \geq 0}$ is a homogeneous Markov chain, i.e., for any $B \in \mathfrak{B}(C)$ and $\xi \in C$
  \begin{align}\label{TH}
      \mathbb{P}\big(~\overline{X}_{t_{k+1}} \in B |~\overline{X}_{t_k} = \xi  \big)= \mathbb{P}\big(~\overline{X}_{t_1} \in B |~\overline{X}_{0}= \xi \big),
  \end{align}
  and 
  \begin{align}\label{MP}
       \mathbb{P}\big(~\overline{X}_{t_{k+1}} \in B |~\mathcal{F}_{t_k} \big)= \mathbb{P}\big(~\overline{X}_{t_{k+1}} \in B |~\overline{X}_{t_k} \big).
  \end{align}
 \end{lemma}
 \begin{proof}
    For any $\theta \in [t_i, t_{i+1}], i = -M, -M+1, ... ,-1$, it follows from \eqref{CDSDDE} and \eqref{GI} that
    \begin{align*}
    \begin{split}
 &\quad\overline{X}(t_1 +\theta) \\
 & =\left \{
 \begin{array}{ll}
G_{t_1+\theta }^{-1}\Big(F_{t_1+\theta }(\overline{X}(t_{1-M})) + \overline{X}(0)+ g(\overline{X}(t_{0}), \overline{X}(t_{-M})(W(t_1+\theta)-W(0))\Big), \quad \theta \in [t_{-1}, 0]\\
\overline{X}(t_1 + \theta), \quad \quad \quad \theta \in [t_i, t_{i+1}], ~i \neq -1,
 \end{array}
 \right. \\
 & =\left \{
 \begin{array}{ll}
G_{t_1+\theta }^{-1}\Big(F_{t_1+\theta }(\xi(t_{1-M})) + \xi(0)+ g(\xi(t_{0}), \xi(t_{-M})(W(t_1+\theta)-W(0))\Big), \quad \theta \in [t_{-1}, 0]\\
\xi(t_1 + \theta), \quad \quad \quad \theta \in [t_i, t_{i+1}], ~i \neq -1,
 \end{array}
 \right. 
 \end{split}
 \end{align*}
  and
  \begin{align*}
    \begin{split}
 &\quad\overline{X}(t_{k+1} +\theta) \\
& =\left \{
 \begin{array}{ll}
G_{t_1+\theta }^{-1}\Big(F_{t_1+\theta }(\overline{X}(t_{k+1-M})) + \overline{X}(t_k)+ g(\overline{X}(t_{k}), \overline{X}(t_{k-M})(W(t_{k+1}+\theta)-W(t_k))\Big), \quad \theta \in [t_{-1}, 0]\\
\overline{X}(t_{k+1} + \theta), \quad \quad \quad \theta \in [t_i, t_{i+1}], ~i \neq -1,
 \end{array}
 \right. \\ 
  & =\left \{
 \begin{array}{ll}
G_{t_1+\theta }^{-1}\Big(F_{t_1+\theta }(\xi(t_{1-M})) + \xi(0)+ g(\xi(t_{0}), \xi(t_{-M})(W(t_{k+1}+\theta)-W(t_k))\Big), \quad \theta \in [t_{-1}, 0]\\
\xi(t_1 + \theta), \quad \quad \quad \theta \in [t_i, t_{i+1}], ~i \neq -1,
 \end{array}
 \right. 
 \end{split}
 \end{align*}
 Note that $W(t_1+\theta)-W(0)$ and $W(t_{k+1}+\theta)-W(t_k)$ are identical in probability law. Comparing the two equations above, we know that $\overline{X}_{t_1}$ and $\overline{X}_{t_{k+1}}$ are equal in probability law under $\overline{X}_{t_k} =\overline{X}_{0} = \xi$. Thus, \eqref{TH} hold.

 Next, we will show that \eqref{MP} is true as well. Set
  \begin{align}
    \begin{split}
Y_{t_{k+1}}^{\eta}(\theta)  =\left \{
 \begin{array}{ll}
G_{t_1+\theta }^{-1}\Big(F_{t_1+\theta }(\eta(t_{1-M})) + \eta(0)+ g(\eta(t_{0}), \eta(t_{-M})(W(t_{k+1}+\theta)-W(t_k))\Big), \quad \theta \in [t_{-1}, 0]\\
\eta(t_1 + \theta), \quad \quad \quad \theta \in [t_i, t_{i+1}], ~i \neq -1,
\end{array}
 \right.
 \end{split}
 \end{align}
 and $Y_{t_{k+1}}^{\eta}=\{Y_{t_{k+1}}^{\eta}(\theta) : -\tau \leq \theta \leq 0\}\} $, It is easy to see that $\overline{X}_{t_{k+1}} = Y_{t_{k+1}}^{\overline{X}_{t_k}}$. For any $0 \leq s \leq t$, let $\mathcal{G}_{t,s} = \sigma (W(u)-W(s), s\leq u \leq t) \wedge \mathcal{N}$. Obviously, $\mathcal{G}_{t_{k+1},t_k}$ is independent of $\mathcal{F}_{t_k}$. Moreover, $Y_{t_{k+1}}^{\eta}$ depends completely on the increment $W(t_{k+1}) - W(t_k)$, so is $\mathcal{G}_{t_{k+1},t_k}-$ measurable. Thus, $Y_{t_{k+1}}^{\eta}$ is independent of $\mathcal{F}_{t_k}$. Then by applying a classical result (see, for example, lemma 9.2 on p.87 of \cite{M2007}) yields
 \begin{align*}
    \mathbb{P}(~\overline{X}_{t_{k+1}} \in B |~\mathcal{F}_{t_k}) &= \mathbb{E} \big(\boldsymbol{1}_B(Y_{t_{k+1}}^{\overline{X}_{t_k}})\big|~\mathcal{F}_{t_k}\big) = \mathbb{E} \big(\boldsymbol{1}_B(Y_{t_{k+1}}^{\eta)}\big)\big|_{\eta =\overline{X}_{t_k}} \\
    &=\mathbb{P}\big(Y_{t_{k+1}}^{\overline{X}_{t_k}} \in B\big)\big|_{\eta =\overline{X}_{t_k}}
    = \mathbb{P}\big(~\overline{X}_{t_{k+1}} \in B |~\overline{X}_{t_k} \big)
 \end{align*}
 The proof is complete.
 \end{proof}
 \subsection{The second and third conditions in Condition \ref{con19}}
 In the following, we will analyze the second and third conditions described in Condition \ref{con19}. The first is the moment boundedness of the true segment process $\{x_t^{0,\xi}\}_{t\geq 0}$.
 \begin{lemma}\label{L1}
     Let Assumptions \ref{as1}  and \ref{as2} hold, then for any $t > 0$, the segment process $\{x_t^{0,\xi}\}_{t \geq 0}$ satisfies
    \begin{equation*}
         \mathbb{E} \|x_t^{0, \xi} \|^2 \leq C_1 \mathbb{E} \|\xi \|^2 e^{-\lambda t} + C_2,
    \end{equation*}
    where $C_1, C_2, \lambda$ are positive constants.
 \end{lemma}
 \begin{proof}
     First, let $\lambda :=b_3- b_4$, then by the It$\rm{\hat{o}}$ formula, we have
     \begin{equation}\label{L1-1}
         \begin{aligned}
             e^{\lambda t}|x(t)|^2&=|\xi(0)|^2 + \int_0^t \lambda e^{\lambda s}|x(s)|^2 + e^{\lambda s}\big(2\langle x(s), f(x(s), x(s-\tau))\rangle + |g(x(s), x(s-\tau))|^2\big) ds \\
             & \quad + 2\int_0^t e^{\lambda s}\langle x(s), g(x(s), x(s-\tau)) \rangle dW(s).
         \end{aligned}
     \end{equation}
     Taking expectations on both sides of \eqref{L1-1} and applying Assumption \ref{as2} yields
     \begin{align}\label{L1-2}
           \mathbb{E}\Big(e^{\lambda t}|x(t)|^2\Big)&\leq\mathbb{E}|\xi(0)|^2 + \mathbb{E}\int_0^t \lambda e^{\lambda s}|x(s)|^2 + e^{\lambda s}\big(-b_3|x(s)|^2 + b_4|x(s-\tau)|^2 +b_5\big) ds  \\ \notag
           & = \mathbb{E}|\xi(0)|^2 + \mathbb{E}\int_0^t (\lambda-b_3+b_4) e^{\lambda s}|x(s)|^2ds + b_4\mathbb{E}\int_{-\tau}^0 e^{\lambda s} |x(s)|^2ds \\ \notag
           &\quad- b_4\mathbb{E}\int_{t-\tau}^t  e^{\lambda s} |x(s)|^2ds +b_5\int_0^t  e^{\lambda s} ds
     \end{align}
     Since $\lambda =b_3- b_4$, this implies that
     \begin{align}\label{L1-3}
          \mathbb{E}\Big(|x(t)|^2\Big)
           & \leq  e^{-\lambda t}\mathbb{E}|\xi(0)|^2 + \frac{b_4 \tau e^{-\lambda t}}{\lambda}\mathbb{E} \|\xi \|^2+ \frac{b_5}{\lambda} \\ \notag
           & \leq  e^{-\lambda t} \Big(1+\frac{b_4 \tau}{\lambda}\Big)\mathbb{E} \|\xi \|^2+ \frac{b_5}{\lambda}.
     \end{align}
   By using the It$\rm{\hat{o}}$ formula again, we can see that for any $t > \tau$ and $\theta \in [0, \tau]$
     \begin{align}\label{L1-4}
         |x(t-\theta)|^2 &=|x(t-\tau)|^2 + \int_{t-\tau}^{t-\theta } \big(2\langle x(s), f(x(s), x(s-\tau))\rangle + |g(x(s), x(s-\tau))|^2\big) ds \\ \notag
         &\quad + 2 \int_{t-\tau}^{t-\theta} \langle x(s), g(x(s), x(s-\tau)) \rangle dW(s).
     \end{align}
     Then, by the Burkholder-Davis-Gundy and Young inequalities, we get
     \begin{align}\label{L1-5}
         &\quad \mathbb{E} \Bigg(\sup \limits_{0 \leq \theta \leq \tau} \Big|\int_{t-\tau}^{t-\theta} \langle x(s), g(x(s), x(s-\tau)) \rangle dW(s)\Big| \Bigg)\\ \notag
         & \leq \mathbb{E} \Bigg(\Big( \sup \limits_{0 \leq \theta \leq \tau} |x(t-\theta)|\Big)\Big(\sup \limits_{0 \leq \theta \leq \tau}\int_{t-\tau}^{t-\theta}|g(x(s), x(s-\tau))| dW(s)\Big)\Bigg) \\ \notag
         &\leq  \frac{1}{4}\mathbb{E} \Big( \sup \limits_{0 \leq \theta \leq \tau} |x(t-\theta)|^2\Big)+ \mathbb{E}\Big(\sup \limits_{0 \leq \theta \leq \tau}\Big|\int_{t-\tau}^{t-\theta}|g(x(s), x(s-\tau))| dW(s)\Big|^2\Big) \\ \notag
          &\leq  \frac{1}{4}\mathbb{E} \Big( \sup \limits_{0 \leq \theta \leq \tau} |x(t-\theta)|^2\Big)+ 4\mathbb{E}\Big(\int_{t-\tau}^{t}|g(x(s), x(s-\tau))|^2 ds\Big).
     \end{align}
     Substituting \eqref{L1-5} into \eqref{L1-4} and it follows from Assumptions \ref{as1} and \ref{as2}  that
     \begin{align}
        &\quad \mathbb{E} \Big( \sup \limits_{0 \leq \theta \leq \tau} |x(t-\theta)|^2\Big) \\ \notag
         &\leq 2\mathbb{E} |x(t-\tau)|^2 + 2\mathbb{E} \Big( \sup \limits_{0 \leq \theta \leq \tau}\int_{t-\tau}^{t-\theta } \big(-b_3|x(s)|^2 + b_4|x(s-\tau)|^2+b_5\big) ds \Big) \\ \notag
         &\quad +8\mathbb{E}\Big(\int_{t-\tau}^{t}\big(c_1(|x(s)|^2+|x(s-\tau)|^2) +c_2 \big) ds\Big) \\ \notag
         &\leq 2\mathbb{E} |x(t-\tau)|^2+8c_1\mathbb{E}\Big(\int_{t-\tau}^{t}|x(s)|^2 ds\Big) +(2b_4+8)\mathbb{E}\Big(\int_{t-\tau}^{t}|x(s-\tau)|^2 \big) ds\Big) + 2\tau(b_5+4c_2)
     \end{align}
     This, together with \eqref{L1-3} implies that
    \begin{equation}
        \begin{aligned}
             &\quad \mathbb{E} \Big( \sup \limits_{0 \leq \theta \leq \tau} |x(t-\theta)|^2\Big) \\ \notag
         &\leq e^{-\lambda (t-\tau)} \Big(2+\frac{2b_4 \tau}{\lambda}\Big)\mathbb{E} \|\xi \|^2+ \frac{2b_5}{\lambda}
         + 8c_1 \Big(1+\frac{b_4 \tau}{\lambda}\Big)\mathbb{E} \|\xi \|^2 \frac{e^{-\lambda t}(e^{\lambda \tau}-1)}{\lambda}+\frac{8c_1b_5\tau}{\lambda} \\ \notag
         &\quad + (2b_4+8) \Big(1+\frac{b_4 \tau}{\lambda}\Big)\mathbb{E} \|\xi \|^2 \frac{e^{-\lambda t}(e^{2\lambda \tau}-e^{\lambda \tau})}{\lambda}+\frac{(2b_4+8)b_5\tau}{\lambda} + 2\tau (b_5+4c_2).
        \end{aligned}
    \end{equation}
     Let 
     \begin{align}
         C_1&:= \Big(2e^{\lambda \tau} + \frac{8c_1(e^{\lambda \tau}-1)}{\lambda} + \frac{(2b_4+8)(e^{2\lambda \tau}-e^{\lambda \tau})}{\lambda} \Big)\Big(1+ \frac{b_4\tau}{\lambda}\Big),\\
         C_2&:= \frac{2b_5}{\lambda} +\frac{8c_1b_5\tau}{\lambda} +\frac{(2b_4+8)b_5\tau}{\lambda} + 2\tau (b_5+4c_2).
     \end{align}
     This completes the proof. 
 \end{proof}

Similar to the proof of Lemma \ref{L1}, we can also establish asymptotic attraction of $\{x_t^{0,\xi}\}_{t\geq 0}$. We state the result as follows and omit the proof.
 \begin{lemma}\label{L-6}
    Suppose that Assumptions \ref{as1}  and \ref{as2} hold, then for any $t > 0$, the segment processes associated with SDDE \ref{SDDE} satisfy
    \begin{equation*}
         \mathbb{E} \|x_t^{0, \xi}-x_t^{0,\eta} \|^2 \leq C_3 \mathbb{E} \|\xi -\eta\|^2 e^{-\epsilon t} ,
    \end{equation*}
    where $C_3, \epsilon$ are positive constants and $\xi, \eta \in C_{\mathcal F_0}^{2}$.
    \end{lemma}

\subsection{The fourth condition in Condition \ref{con19}}
In this subsection, we will analyze the strong convergence of the numerical segment process in arbitrary finite time, but before doing so, some lemmas need to be established first.
\begin{lemma}\label{L6}
    Under Assumption \ref{as1}, for any $T_2>T_1 \geq 0$, the numerical solution \eqref{DSDDE} satisfies
    \begin{align}
        \sup \limits_{(N_1-2M) \vee 0 \leq k \leq N_2} \mathbb{E}|X_k|^2 \leq C_4^{T_2-T_1} \big(\sup \limits_{(T_1-\tau) \vee 0 \leq t_k \leq T_1}\mathbb{E}\|\overline{X}_{t_k}\|^2\big) +C_5^{T_2-T_1}
    \end{align}
   where $C_4^{T_2-T_1}$ and $C_5^{T_2-T_1}$ are positive constant dependent on $T_2-T_1$.
   \begin{proof}
     First, for simplicity, denote $X_k^{0,\xi}$ by $X_k$, then by Assumption \ref{as1} and the Young inequality, it follows from \eqref{DSDDE} that for any $N_1+1\leq k \leq N_2$
       \begin{align*}
         &\quad |X_k|^2 - |X_{k-1}|^2 + |X_{k}-X_{k-1}|^2 \\ \notag
         &= 2\langle X_k-X_{k-1}, X_k\rangle \\ \notag
         &=2 \langle f(X_k, X_{k-M}), X_k \rangle\Delta + 2 \langle g(X_{k-1}, X_{k-M-1})\Delta W_{k-1}, X_k \rangle \\ \notag
         &=2 \langle f(X_k, X_{k-M}), X_k \rangle\Delta + 2 \langle g(X_{k-1}, X_{k-M-1})\Delta W_{k-1}, X_k-X_{k-1} \rangle + 2 \langle g(X_{k-1}, X_{k-M-1})\Delta W_{k-1}, X_{k-1} \rangle  \\ \notag
         & \leq (-b_3|X_k|^2 +b_4|X_{k-M}|^2 + b_5)\Delta - |g(X_k, X_{k-M})|^2\Delta + | g(X_{k-1}, X_{k-M-1})\Delta W_{k-1}|^2 + |X_k-X_{k-1}|^2 \\ \notag
         &\quad+  2 \langle g(X_{k-1}, X_{k-M-1})\Delta W_{k-1}, X_{k-1} \rangle 
       \end{align*}
       Then, canceling the same terms on both sides yields
       \begin{align*}
           (1+b_3\Delta)|X_k|^2 &\leq |X_{k-1}|^2 + |g(X_{k-1},X_{k-M-1})\Delta W_{k-1}|^2 +b_4 |X_{k-M}|^2\Delta +b_5\Delta \\ \notag
           &\quad+  2 \langle g(X_{k-1}, X_{k-M-1})\Delta W_{k-1}, X_{k-1} \rangle 
       \end{align*}
       Next, taking the expectation and making use of \eqref{LG}, we can see that 
       \begin{align*}
           \mathbb{E}|X_k|^2 & \leq \frac{1+c_1\Delta}{1+b_3\Delta} \mathbb{E}|X_{k-1}|^2+ \frac{c_1\Delta}{1+b_3\Delta}\mathbb{E}|X_{k-M-1}|^2 +\frac{b_4\Delta}{1+b_3\Delta}\mathbb{E}|X_{k-M}|^2 +\frac{(c_2+b_5)\Delta}{1+b_3\Delta} \\ \notag
           & \leq (1+c_1\Delta)\mathbb{E}|X_{k-1}|^2 + c_1\Delta \mathbb{E}|X_{k-M-1}|^2 + b_4\Delta \mathbb{E}|X_{k-M}|^2 + (c_2+b_5)\Delta \\ \notag
           & \quad\vdots \\ \notag
          & \leq (1+c_1\Delta)^{k-N_1} \mathbb{E}|X_{N_1}|^2+ c_1\Delta \sum \limits_{i=N_1-M}^{k-M-1} (1+c_1\Delta)^{k-M-1-i} \mathbb{E}|X_{i}|^2  \\  \notag
          &\quad+ b_4\Delta \sum \limits_{i=N_1-M}^{k-M-1} (1+c_1\Delta)^{k-M-1-i} \mathbb{E}|X_{i+1}|^2 + (c_2+b_5)\Delta \sum \limits_{i=0}^{k-N_1-1}  (1+c_1\Delta)^i . 
     \end{align*}
Note that $X_k=\overline{X}(kh)$, then by Lemma \ref{L--5} and the time-homogeneous of $\{\overline{X}_{t_k}\}_{k \geq 0}$, we also have
     \begin{align}\label{L-19-1}
           \mathbb{E}|X_k|^2 = \mathbb{E}|X_{k-N_1}^{0, X_{N_1}}|^2 
           &\leq (1+c_1\Delta)^{k-N_1} \mathbb{E}|X_{N_1}|^2+ c_1\Delta \sum \limits_{i=-M}^{k-N_1-M-1} (1+c_1\Delta)^{k-N_1-M-1-i} \mathbb{E}|X_{i}^{0, X_{N_1}}|^2 \\  \notag
          &\quad+ b_4\Delta \sum \limits_{i=-M}^{k-N_1-M-1} (1+c_1\Delta)^{k-N_1-M-1-i} \mathbb{E}|X_{i+1}^{0, X_{N_1}}|^2  + (c_2+b_5)\Delta \sum \limits_{i=0}^{k-N_1-1}  (1+c_1\Delta)^i  
     \end{align}
       Since
        \begin{align*}
            \lim \limits_{\Delta \to 0}   (1+c_1 \Delta)^{N_2-N_1} =  \lim \limits_{\Delta \to 0}   (1+c_1 \Delta)^{\frac{t_2-t_1}{\Delta}}=e^{c_1(T_2-T_1)},
        \end{align*}
       there exists a positive constant $\Delta^* \in [0,1)$ such that for any $\Delta \in [0, \Delta^*)$
       \begin{align}\label{L-19-2}
          (1+c_1 \Delta)^{N_2-N_1} \leq e^{2c_1(T_2-T_1)}.
       \end{align}

Note that for $N_1+M+1 \leq k \leq N_2$, it follows from \eqref{L-19-1} and \eqref{L-19-2} that
 \begin{align*}
           \mathbb{E}|X_k|^2 &= \mathbb{E}|X_{k-N_1}^{0, X_{N_1}}|^2  \\ \notag
           &\leq e^{2c_1(T_2-T_1)} \mathbb{E}|X_{N_1}|^2 + (c_1 + b_4) e^{2c_1(T_2-T_1)} \Delta\sum \limits_{i=-M}^{k-N_1-M-1} \mathbb{E}|X_i^{0, X_{N_1}}|^2 + 
        \frac{e^{2c_1(T_2-T_1)}  (c_2+b_5)}{c_1} \\ \notag
        &\leq \frac{e^{2c_1(T_2-T_1)}  (c_2+b_5)}{c_1} + \Big(e^{2c_1(T_2-T_1)}+(c_1 + b_4) \tau e^{2c_1(T_2-T_1)} \Big) \Big(\mathbb{E}\big(\sup \limits_{-M \leq i \leq -1} |X_i^{0, X_{N_1}}|^2 \big)\Big) \\ \notag
        &\quad+ (c_1 + b_4) e^{2c_1(T_2-T_1)} \Delta\sum \limits_{i=0}^{k-N_1-M-1} \mathbb{E}|X_i^{0,X_{N_1}}|^2. 
       \end{align*}
Similarly, for $N_1+1 \leq k \leq N_1+M$, we get
  \begin{align*}
           \mathbb{E}|X_k|^2 &= \mathbb{E}|X_{k-N}^{0, X_N}|^2  \\ \notag
           &\leq e^{2c_1(T_2-T_1)} \mathbb{E}|X_{N}|^2 + (c_1 + b_4) e^{2c_1(T_2-T_1)} \Delta\sum \limits_{i=-M}^{k-N-M-1} \mathbb{E}|X_i^{0, X_N}|^2 + 
        \frac{e^{2c_1(T_2-T_1)}  (c_2+b_5)}{c_1} \\ \notag
        &\leq \frac{e^{2c_1(T_2-T_1)}  (c_2+b_5)}{c_1} + \Big(e^{2c_1(T_2-T_1)}+(c_1 + b_4) \tau e^{2c_1(T_2-T_1)} \Big) \Big(\mathbb{E}\big(\sup \limits_{-M \leq i \leq 0} |X_i^{0, X_N}|^2 \big)\Big)
        \end{align*}
 
      Thus, for $N_1+1 \leq k \leq N_2$,  recall that $\overline{X}_{T}^{0, \xi} = \{\overline{X}(T-\theta; 0, \xi): -\tau\leq \theta\leq 0\}$ and $\overline{X}(t_k; 0,\xi)= X_k^{0, \xi}$, then we can apply the discrete-type Gronwall inequality such that
       \begin{align}\label{L-19-3}
            \mathbb{E}|X_k|^2 &= \mathbb{E}|X_{k-N}^{0, X_N}|^2  \\ \notag
            &\leq \Big(\frac{e^{2c_1(T_2-T_1)}  (c_2+b_5)}{c_1} + \Big(e^{2c_1(T_2-T_1)}+(c_1 + b_4) \tau e^{2c_1(T_2-T_1)} \Big) \Big(\mathbb{E}\big(\sup \limits_{-M \leq i \leq 0} |X_i^{0, X_N}|^2 \big)\Big)\Big)\\ \notag
            &\quad \times e^{2(c_1+b_4)(T_2-T_1)e^{2c_1(T_2-T_1)}} \\ \notag
             &= \Big(\frac{e^{2c_1(T_2-T_1)}  (c_2+b_5)}{c_1} + \Big(e^{2c_1(T_2-T_1)}+(c_1 + b_4) \tau e^{2c_1(T_2-T_1)} \Big) \Big(\mathbb{E}\big(\sup \limits_{-M \leq i \leq 0} |X_{N+i}^{0,\xi}|^2 \big)\Big)\Big)\\ \notag
            &\quad \times e^{2(c_1+b_4)(T_2-T_1)e^{2c_1(T_2-T_1)}} \\ \notag
            & \leq \Big(\frac{e^{2c_1(T_2-T_1)}  (c_2+b_5)}{c_1} + \Big(e^{2c_1(T_2-T_1)}+(c_1 + b_4) \tau e^{2c_1(T_2-T_1)} \Big) \Big(\mathbb{E}\big(\|\overline{X}_{T_1}^{0,\xi}\|^2\big)\Big)\Big)\\ \notag
            &\quad \times e^{2(c_1+b_4)(T_2-T_1)e^{2c_1(T_2-T_1)}}
       \end{align}
       As for any $(N_1-2M) \vee 0 \leq k \leq N_1$, it is obvious that
       \begin{align}\label{L-19-4}
           \mathbb{E}|X_k|^2 \leq \sup \limits_{(T_1-\tau) \vee 0 \leq t_k \leq T_1} \mathbb{E}\|\overline{X}_{t_k}^{0,\xi}\|^2
       \end{align}
     Combining \eqref{L-19-3} and \eqref{L-19-4} and setting
     \begin{align*} 
     C_4^{T_2-T_1}&:= \big(1+(c_1+b_4)\tau\big)e^{2c_1(T_2-T_1)+ 2(c_1+b_4)(T_2-T_1)e^{2c_1(T_2-T_1)} } \\ \notag
       C_5^{T_2-T_1}&:=\frac{ c_2+b_5}{c_1}e^{2c_1(T_2-T_1)+ 2(c_1+b_4)(T_2-T_1)e^{2c_1(T_2-T_1)} }.
     \end{align*}
     The proof is complete.
   \end{proof}
 \end{lemma}
 
 \begin{lemma}\label{L-14}
     Suppose that Assumption \ref{as1} holds, then for any $T_1 \geq 0$ and $t \in [T_1-\tau, T_2]$,
\begin{align}
    \mathbb{E}\big(|\overline{X}(t)- X_1(t)|^2\big) \leq C_6^{T_2-T_1} \Big(\big(\sup \limits_{(T_1-\tau) \vee 0 \leq t_k \leq T_1}\mathbb{E}\|\overline{X}_{t_k}\|^2\big) +1\Big)\Delta,  \label{L-14-1}\\ 
     \mathbb{E}\big(|\overline{X}(t)- X_2(t)|^2\big) \leq C_7^{T_2-T_1} \Big(\big(\sup \limits_{(T_1-\tau) \vee 0 \leq t_k \leq T_1}\mathbb{E}\|\overline{X}_{t_k}\|^2\big) +1\Big)\Delta. \label{L-14-2},
\end{align}
where $C_6^{T_2-T_1}$ and $C_7^{T_2-T_1}$ are a pair of constant dependent on $T_2-T_1$.
 \end{lemma}
 \begin{proof}
     First, for any $t \geq (T_1-\tau) \vee 0$, choose $k$ for $t \in [t_k, t_{k+1})$, then by \eqref{CDSDDE}, we have
     \begin{align*}
         \overline{X}(t)-X_1(t)=\int_t^{t_{k+1}} f(X_1(s), X_1(s-\tau))ds + \int_t^{t_{k+1}} g(X_2(s), X_2(s-\tau)) dW(s).
     \end{align*}
 Next, using the It$\rm{\hat{o}}$ isometry and \eqref{LG}, we obtain
  \begin{align*}
      \mathbb{E}(|\overline{X}(t)- X_1(t)|^2) &\leq 2\Delta \mathbb{E}\int_t^{t_{k+1}} |f(X_1(s), X_1(s-\tau))|^2ds + 2\mathbb{E}\int_t^{t_{k+1}} |g(X_2(s), X_2(s-\tau))|^2 ds \\ \notag
      &\leq 2\tau \Big(c_1 (\mathbb{E}|X_{k+1}|^2 + \mathbb{E}|X_{k+1-M}|^2 )+c_2\Big)\Delta+ 2\Big(c_1 (\mathbb{E}|X_{k}|^2 + \mathbb{E}|X_{k-M}|^2 )+c_2\Big)\Delta.
  \end{align*}
   For any $t \in [(T_1-\tau) \vee 0,T_2]$, this together with Lemma \ref{L6} implies
  \begin{align*}
      \mathbb{E}(|\overline{X}(t)- X_1(t)|^2) &\leq 4c_1(\tau +1) \Big(C_4^{T_2-T_1} \big(\sup \limits_{(T_1-\tau) \vee 0 \leq t_k \leq T_1}\mathbb{E}\|\overline{X}_{t_k}\|^2\big)+ C_5^{T_2-T_1} \Big)\Delta +2c_2(\tau+1)\Delta  \\ \notag
      &\leq C_6^{T_2-T_1} \Big(\big(\sup \limits_{T_1-\tau \leq t_k \leq T_1}\mathbb{E}\|\overline{X}_{t_k}\|^2\big) +1\Big)\Delta,
  \end{align*}
  Meanwhile, by Assumption \ref{as3}, for any $t \in [-\tau, 0]$, we also have
  \begin{align}
      \mathbb{E}(|\overline{X}(t)- X_1(t)|^2) \leq K_1 \Delta.
  \end{align}
 Thus, let $C_6^{T_2-T_1} := K_1 \vee \Big(4c_1(\tau +1)C_4^{T_2-T_1} \Big) \vee \Big(4c_1(\tau +1)C_5^{T_2-T_1}  + 2c_2(\tau+1)\Big)$. Then, we get the required assertion \eqref{L-14-1}, and the proofs of \eqref{L-14-2} follow a similar course as previously outlined. 
 \end{proof}

    \begin{lemma}\label{L11}
       Suppose that Assumptions \ref{as1} and \ref{as2} hold, then for any $\Delta \in (0,1)$, $T_2 > T_1 \geq 0$ and $t \in [T_1, T_2]$
        \begin{align*}
            \sup \limits_{T_1 \leq t_k \leq T_2} \mathbb{E}\Big(\|x_{t_k}^{T_1, \overline{X}_{T_1}}-\overline{X}_{t_k}^{0,\xi}\|^2\Big) \leq C^{T_2-T_1}_8 \Big(\sup \limits_{T_1-\tau \leq t_k\leq T_1}\mathbb{E}\|\overline{X}_{t_k}^{0,\xi}\|^2 +1\Big)\Delta,
        \end{align*}
        where $C^{T_2-T_1}_8$ is a positive constant depneds on $T_2-T_1$.
    \end{lemma}
  \begin{proof}
       First, for simplicity, write $y(t):= x(t;T_1, \overline{X}_{T_1}^{0,\xi})$. From \eqref{SDDE} and \eqref{CDSDDE}, for any $t \in [T_1, T_2]$, we have
        \begin{align}\label{L14-1}
            y(t)- \overline{X}(t)&=\int_{T_1 }^t f(y(s), y(s-\tau)) - f(X_1(s), X_{1}(s-\tau)) ds \\ \notag
            &\quad+ \int_{T_1}^{t}g(y(s), y(s-\tau))-g(X_2(s), X_{2}(s-\tau))dW(s)
        \end{align}
        Then by the H$\rm{\ddot{o}}lder$ inequality, the Burkholder–Davis–Gundy inequality and Assumption \ref{as1}, it follows from \eqref{L14-1} that
        \begin{align}\label{L11-1}
          &\quad ~ \mathbb{E}\Big(\sup \limits_{T_1  \leq s \leq T_2}|y(s)-\overline{X}(s)|^2\Big)\\ \notag
            &\leq 2(T_2-T_1)\mathbb{E} \Big(\int_{T_1}^{T_2} |f(y(s), y(s-\tau)) - f(X_1(s), X_{1}(s-\tau))|^2 ds\Big) \\ \notag
            & \quad+8\mathbb{E}\Big(\int_{T_1}^{T_2}|g(y(s), y(s-\tau))-g(X_2(s), X_{2}(s-\tau))|^2ds\Big) \\ \notag
            &\leq 2a_1(T_2-T_1)\mathbb{E}\Big(\int_{T_1}^{T_2}(|y(s)-X_1(s)|^2 + |y(s-\tau)-X_1(s-\tau)|^2)ds\Big)\\ \notag
            & \quad+8a_1\mathbb{E}\Big(\int_{T_1}^{T_2} (|y(s)-X_2(s)|^2 + |y(s-\tau)-X_2(s-\tau)|^2)ds\Big)
        \end{align}
        Next, let's analyze the above inequality. By applying the elementary inequality $2ab \leq |a|^2+|b|^2$, we can see that
        \begin{align}\label{L11-2}
           &\quad \mathbb{E}\Big(\int_{T_1}^{T_2}(|y(s)-X_1(s)|^2 + |y(s-\tau)-X_1(s-\tau)|^2)ds\Big)  \\ \notag
           & \leq 2 \mathbb{E}\Big(\int_{T_1}^{T_2}(|y(s)-\overline{X}(s)|^2+|\overline{X}(s)-X_1(s)|^2 + |y(s-\tau)-\overline{X}(s-\tau)|^2+|\overline{X}(s-\tau)-X_1(s-\tau)|^2)ds\Big) \\ \notag
           &= J_1+J_2+J_3+J_4. 
        \end{align}
        Note that 
        \begin{align}\label{L11-3}
            J_1=2 \mathbb{E}\Big(\int_{T_1}^{T_2}|y(s)-\overline{X}(s)|^2ds \Big) \leq 2\int_{T_1}^{T_2}\mathbb{E}\Big(\sup \limits_{T_1 \leq u \leq s}|y(u)-\overline{X}(u)|^2\Big)ds 
        \end{align}
        Meanwhile, from Lemma \ref{L-14}, we have
        \begin{align}\label{L11-4}
            J_2 &=2\mathbb{E}\Big(\int_{T_1}^{T_2}|\overline{X}(s)-X_1(s)|^2ds \Big) \\ \notag
            &\leq 2(T_2-T_1)C_6^{T_2-T_1}\Big(\big(\sup \limits_{(T_1-\tau) \vee 0 \leq t_k \leq T_1}\mathbb{E}\|\overline{X}_{t_k}\|^2\big)+1\Big)\Delta.
        \end{align}
        Similarly,
        \begin{align}\label{L11-5}
            J_3 &= 2 \mathbb{E}\Big(\int_{T_1}^{T_2}|y(s-\tau)-\overline{X}(s-\tau)|^2ds \Big) 
            =2 \mathbb{E}\Big(\int_{T_1}^{T_2-\tau}|y(s)-\overline{X}(s)|^2ds \Big) \\ \notag
            &\leq 2 \mathbb{E}\Big(\int_{T_1}^{T_2}|y(s)-\overline{X}(s)|^2ds \Big) \leq 2\int_{T_1}^{T_2}\mathbb{E}\Big(\sup \limits_{T_1 \leq u \leq s}|y(u)-\overline{X}(u)|^2\Big)ds .
        \end{align}
        As for $J_4$, from Assumption \ref{as3} and Lemma \ref{L-14}, we get 
        \begin{align}\label{L11-6}
    J_4&=2\mathbb{E}\Big(\int_{T_1}^{T_2}|\overline{X}(s-\tau)-X_1(s-\tau)|^2ds \Big) \\ \notag
            &=2\mathbb{E}\Big(\int_{T_1-\tau}^{T_2-\tau}|\overline{X}(s)-X_1(s)|^2ds \Big)  \\ \notag
            &\leq 2(T_2-T_1) (C_6^{T_2-T_1}+K_1)\Big(\big(\sup \limits_{(T_1-\tau)\vee 0 \leq t_k \leq T_1}\mathbb{E}\|\overline{X}_{t_k}\|^2\big)+1\Big)\Delta.
        \end{align}
        Inserting \eqref{L11-3},\eqref{L11-4},\eqref{L11-5} and \eqref{L11-6} into \eqref{L11-2} yields
        \begin{align}\label{L11-7}
           &\quad \mathbb{E}\Big(\int_{T_1}^{T_2}(|y(s)-X_1(s)|^2 + |y(s-\tau)-X_1(s-\tau)|^2)ds\Big)  \\ \notag
           & \leq 2(T_2-T_1)( 2C_6^{T_2-T_1}+K_1)\Big(\big(\sup \limits_{(T_1-\tau)\vee 0 \leq t_k \leq T_1}\mathbb{E}\|\overline{X}_{t_k}\|^2\big)+1\Big)\Delta + 4\int_{T_1}^{T_2}\mathbb{E}\Big(\sup \limits_{T_1 \leq u \leq s}|y(u)-\overline{X}(u)|^2\Big)ds.
        \end{align}
     Similar to the proof used in \eqref{L11-7}, we derive that 
     \begin{align}\label{L11-8}
        &\quad \mathbb{E}\Big(\int_{T_1}^{T_2} (|y(s)-X_2(s)|^2 + |y(s-\tau)-X_2(s-\tau)|^2)ds\Big)  \\    & \leq  2(T_2-T_1)( C_7^{T_2-T_1}+K_1)\Big(\big(\sup \limits_{(T_1-\tau)\vee 0 \leq t_k \leq T_1}\mathbb{E}\|\overline{X}_{t_k}\|^2\big)+1\Big)\Delta + 4\int_{T_1}^{T_2}\mathbb{E}\Big(\sup \limits_{T_1 \leq u \leq s}|y(u)-\overline{X}(u)|^2\Big)ds
     \end{align}
     Then, inseting \eqref{L11-7} and \eqref{L11-8} into \eqref{L11-1}, we obtain
     \begin{align*}
           &\quad\mathbb{E}\Big(\sup \limits_{T_1 \leq s \leq T_2}|y(s)-\overline{X}(s)|^2\Big) \\ \notag
          & \leq \Big(2(T_2-T_1)(2a_1(C_6^{T}+K_1)(T_2-T_1)+8a_1(C_7^{T_2-T_1}+K_1))\Big)\Big (\big(\sup \limits_{(T_1-\tau)\vee 0 \leq t_k \leq T_1}\mathbb{E}\|\overline{X}_{t_k}\|^2\big) + 1\Big) \Delta\\
          &\quad+ 8a_1\Big(T_2-T_1+4\Big)\int_{T_1}^{T_2}\mathbb{E}\Big(\sup \limits_{T_1 \leq u \leq s}|y(u)-\overline{X}(u)|^2\Big)ds. 
     \end{align*}
     Applying the Gronwall inequality yields
     \begin{align}\label{L18-9}
           &\quad~\mathbb{E}\Big(\sup \limits_{T_1 -\tau\leq s \leq T_2}|y(s)-\overline{X}(s)|^2\Big) \\ \notag
           & =\mathbb{E}\Big(\sup \limits_{T_1\leq s \leq T_2}|y(s)-\overline{X}(s)|^2\Big) \\ \notag
          & \leq \Big(2(T_2-T_1)(2a_1(C_6^{T}+K_1)(T_2-T_1)+8a_1(C_7^{T_2-T_1}+K_1))\Big)e^{8a_1(T_2-T_1+4)(T_2-T_1)} \\ \notag
          & \quad \times \Big (\big(\sup \limits_{(T_1-\tau) \vee 0 \leq t_k \leq T_1}\mathbb{E}\|\overline{X}_{t_k}\|^2\big) + 1\Big) \Delta 
     \end{align}
   By the flow property of $\overline{X}(s)$, it not difficult to see that for any $t \in [T_1, T_2]$
     \begin{align}\label{L18-10}
         \mathbb{E}\Big(\|x_t^{T_1,\overline{X}_{T_1}^{0,\xi}}-\overline{X}_t^{T_1,\overline{X}_{T_1}^{T_1,\xi}}\|^2\Big)&\leq\mathbb{E}\Big(\sup \limits_{T_1-\tau \leq s \leq T_1}|x(s; T_1, \overline{X}_{T_1}^{0,\xi})-\overline{X}(s; T_1, \overline{X}_{T_1}^{0,\xi})|^2\Big) \\ \notag
         &= \mathbb{E}\Big(\sup \limits_{T_1-\tau \leq s \leq T_2}|y(s)-\overline{X}(s)|^2\Big).
     \end{align}
     Inserting \eqref{L18-9} into \eqref{L18-10} and letting
     \begin{align*}
         C^{T_2-T_1}_8:=\Big(2(T_2-T_1)(2a_1(C_6^{T}+K_1)(T_2-T_1)+8a_1(C_7^{T_2-T_1}+K_1))\Big)e^{8a_1(T_2-T_1+4)(T_2-T_1)}.
     \end{align*}
     The required assertion follows.
     \end{proof}

From Lemmas \ref{L1}, \ref{L-6}, and \ref{L11}, we can see that Condition \ref{con19} hold, then by Theorems \ref{T-22} and \ref{T-9}, we can conclude this part with the following theorems.
\begin{theorem}\label{T-27}
    Under Assumption \ref{as1}-\ref{as4}, the segment process $\{\overline{X}_{t_k,\xi}\}_{k \geq 0}$ converges strongly to the true segment process $\{x_{t_k}^{0,\xi}\}$ in the infinite horizon.
\end{theorem}

\begin{theorem}\label{T-28}
     Under Assumption \ref{as1}-\ref{as4}, the  probability measure of $\{\overline{X}_{t_k,\xi}\}_{k \geq 0}$ converges to the invariant measure $\pi$ of the segment process $\{x_{t}^{0,\xi}\}_{t \geq 0}$ 
\end{theorem}

\section{Numerical examples}
In this section, some examples will be presented in order to verify the theories.
\begin{example}
    Consider the following SDDE
   \begin{align}
    \begin{split}
  \left \{
 \begin{array}{ll}
dx(t)=\big(-4x(t)+x(t-1)\big)dt + \big(x(t)+x(t-1)+1\big)dW(t), \quad \quad t >0\\
 x(t)=cos(W(t)), \quad \quad t\in [-1,0].
 \end{array}
 \right.
 \end{split}
 \end{align}
 It is easy to check that Assumptions \ref{as1}, \ref{as2} and \ref{as3} hold with $a_1=32, b_1=5, b_2=3, b_3=4.75, b_4=3.25, b_5=9 $ and $K_1 = 1$.

 We approximate the expectation by averaging over 500 paths, taking the numerical solution with the stepsize $\Delta_1=0.0001$ as the true solution $x(t)$ and the stepsize $\Delta_2=0.01$ as the numerical solution $z(t)$, and using Newton iteration to solve the proposed implicit scheme. At each time point $t_k=k\Delta_2$, we also approximate $e_{strong}(t_k):=\mathbb{E}(\|x(t_k)-z(t_k)\|^2)$ by taking the average of $\max_{~0\leq i \leq 100} |x(t_{k-i})-z(t_{k-i})|^2$ across 500 paths. The left figure of Figure \ref{fig5} shows the change of the ratio of error and stepsize $e_{strong}/\Delta_2^2$ over the time interval $[0,100]$. It can be seen that the curve does not rise with the increase of time, but fluctuates within a certain range, which is consistent with the theoretical results.

 And by the definition of the segment process, we know that it's a functionally valued random variable. The distribution of the segment process is composed of the
distribution of an infinite and uncountable number of points. We cannot perform numerical experiments on it, and can only study the distribution of finitely countable points on each interval. Therefore, since $\tau=1$, on each interval $[k, k + 1]$, we only discuss the distribution at $t=k$. Similar to before, we take the probability distribution at $t=100$ of the numerical solution with stepsize $\Delta_1=0.0001$ as the invariant measure of the true solution. At each time point, the probability distribution is approximated by an empirical distribution of $100$ sample points. As shown in Figure \ref{estrong}, on the left we use the Kolmogorov-Smirnov test (K-S test) to measure the difference between the empirical distribution and the invariant measure at each time point. As shown in the figure, the difference tends to be stable, indicating that there is an invariant measure of the true solution. On the right, we draw the empirical distribution function at $t=10,95,96,100$. It is easy to see that the difference between the empirical distributions becomes smaller and smaller as time increases. This is also consistent with the theoretical results.

Finally, due to the close relationship between the invariant measure and ergodicity, we also verify the ergodicity of numerical solutions. As shown in Figure \ref{ ergodicity}, we select functions $f(x)=x^3$ and $f(x)=e^{-x}$, and start from three different initial values $3+cos(W(t))$, $-1+t$, and $-2$. The curves depicting the changes in their time averages $\mathbb{E}(f(z(t)))$ converge together, thus verifying the existence of ergodicity
\end{example}

\begin{example}
    For the truncated EM method, consider another scalar SDDE
    \begin{align}
    \begin{split}
  \left \{
 \begin{array}{ll}
dx(t)=\big(-2x(t)-10 x^3(t)+x(t-1)\big)dt + \big(1+0.5x^2(t)\big)dW(t), \quad \quad t >0\\
 x(t)=cos(W(t)), \quad \quad t\in [-1,0].
 \end{array}
 \right.
 \end{split}
 \end{align}
 \begin{figure}[htp!]
    \centering
    \includegraphics[width=0.49\textwidth]{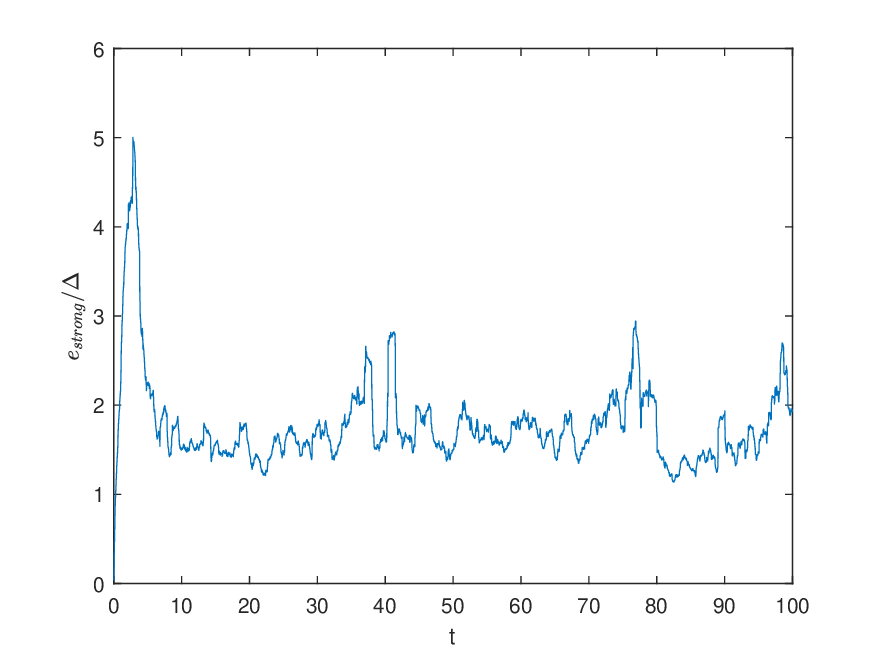}
 \includegraphics[width=0.49\textwidth]{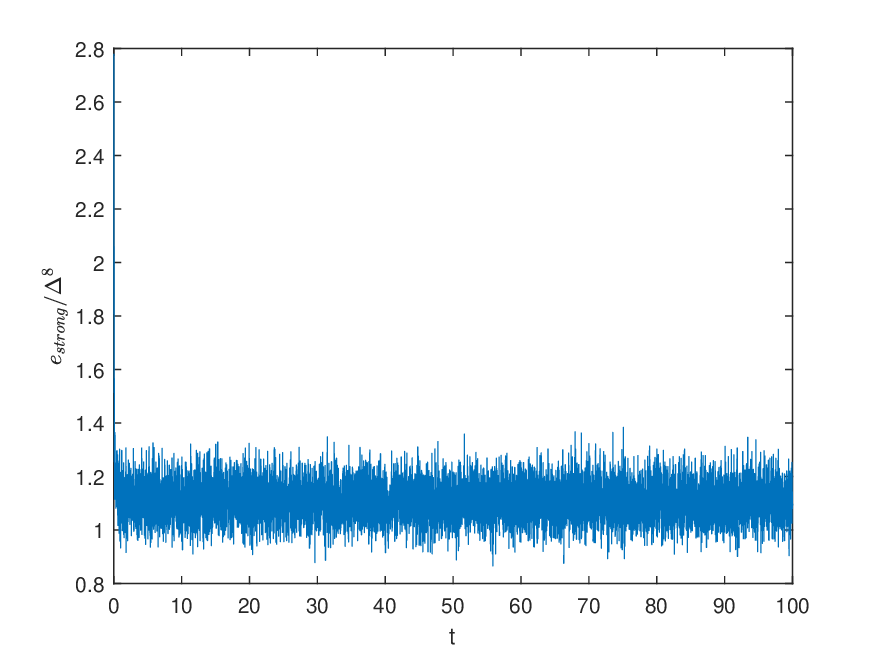}
   \caption{Left: The BEM method; \quad\quad~\quad\quad\quad\quad\quad\quad \quad \quad\quad Right:The TEM method;}
    \label{fig5}
\end{figure}

\begin{figure}[htp!]
    \centering
    \includegraphics[width=0.49\textwidth]{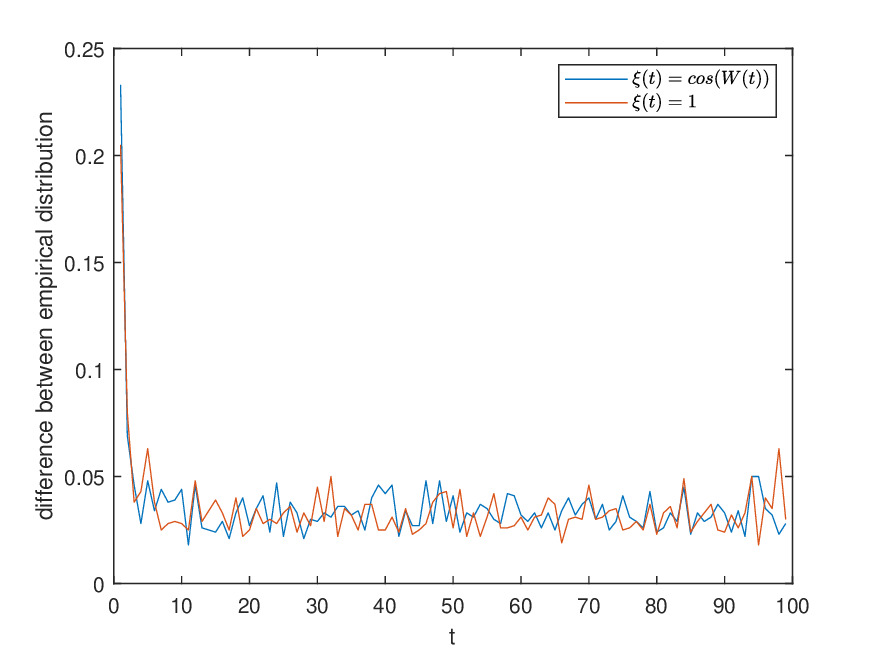}
 \includegraphics[width=0.49\textwidth]{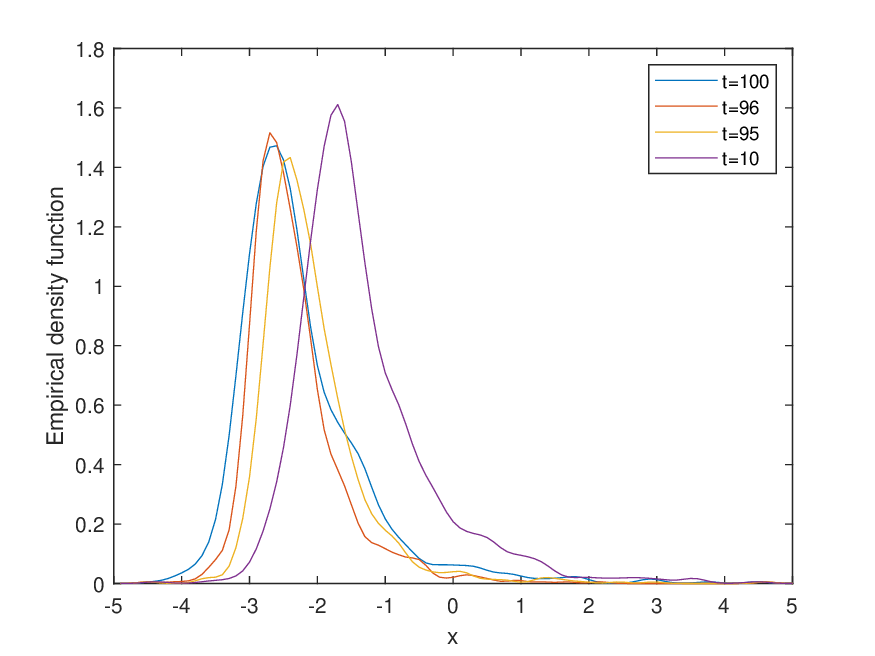}
  \caption{Left:Difference between empirical distribution; ~\quad\quad Right:Empirical density function;}
    \label{estrong}
\end{figure}

\begin{figure}[htp!]
    \centering
    \includegraphics[width=0.49\textwidth]{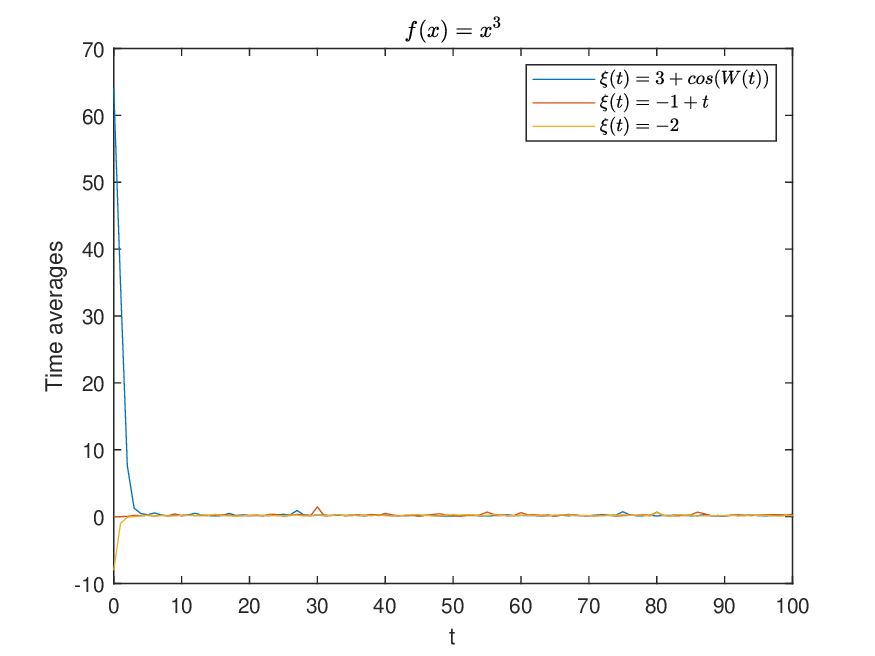}
 \includegraphics[width=0.49\textwidth]{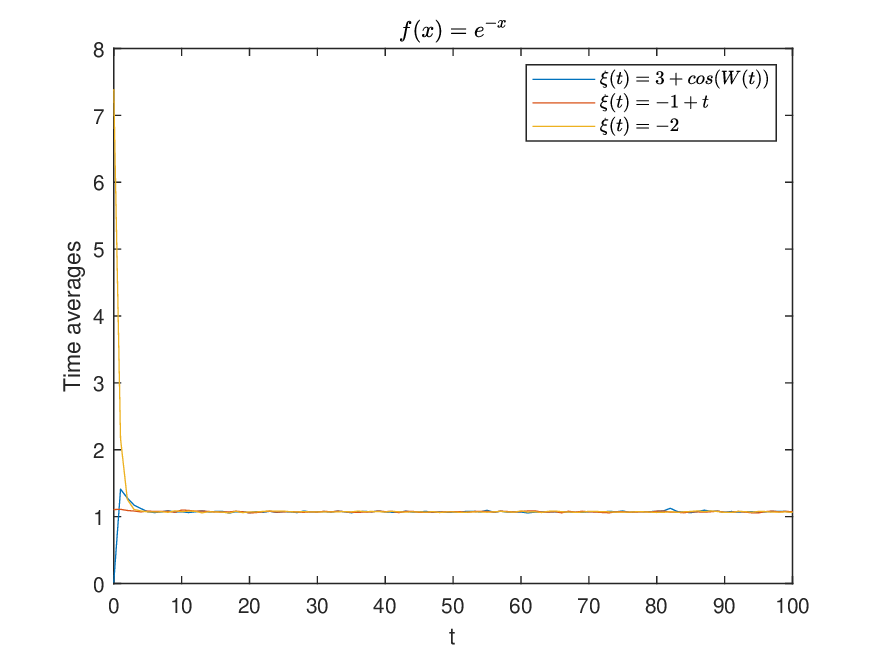}
  \caption{Sample means of $\mathbb{E}(|f(z(k))|)$ $(k\in \mathbb{N})$ with different initial data $\xi$}
    \label{ ergodicity}
\end{figure}
 For any $x,~y,~\overline{x},~\overline{y} \in \mathbb{R}^d$
 \begin{align*}
     &|f(x,y)- f(\overline{x},\overline{y})| \leq 20 (|x-\overline{x}|+ |y-\overline{y}|)(1+|x|^2+|\overline{x}|^2+ |y|^2+ |\overline{y}|^2), \\
    & |g(x,y)-g(\overline{x},\overline{y})|^2 \leq |x-\overline{x}|(|x|^2+|\overline{x}|^2),
 \end{align*}
 and
 \begin{align*}
     &\quad ~\Big(1+|x|^2\Big)^{7}\Big(\langle 2x, f(x,y) \rangle + 15|g(x,y)|^2\Big) \leq 2^7(4-12|x|^{18}+ |y|^{18}), \\
    & \quad \quad \quad \quad \quad\langle 2x, f(x,y) \rangle + 15|g(x,y)|^2 \leq 30-3x^2+y^2 \\ \notag
    & 2\langle x-\overline{x}, f(x,y)-f(\overline{x}, \overline{y}) \rangle  +2|g(x,y)-g(\overline{x},\overline{y})|^2 \leq -3|x-\overline{x}|^2+|y-\overline{y}|^2.
 \end{align*}
 Moreover, for any $s, ~t \in [-1,~0]$,
 \begin{align*}
     \mathbb{E}|cos(W(t))-cos (W(s))|^p \leq \mathbb{E}|W(t)-W(s)|^p \leq K_p (t-s)^{\frac{p}{2}},
 \end{align*}
 where $K_p$ is a positive constant dependent on $p$. It is not difficult to see that Assumptions \ref{as1-1}, \ref{as1-2}, \ref{as1-3} and \ref{as1-4} hold with $a_1=20, q=2, p=16, b_1=512, b_2=1536, b_3=128, b_4=3, b_5=1, \overline{b}_1=30, \overline{b}_2=3$ and $\overline{b}_3=1$.
Then by \eqref{PHI-1}, we take $\phi(r) =421+12\sqrt{20}+ (1568+48\sqrt{20})r^4$, for $r \geq 1$. And by \eqref{PI}, we have $\nu= 2/7$ and $K=120$.

From our previous analysis, we know that the truncated EM numerical solution is strongly convergent in the infinite horizon. Thus, we use the truncated EM method to simulate 1000 sample paths with the stepsize $\Delta_1 =0.001$ as the true solution $x(t)$ and the stepsize $\Delta_2=0.01$ as the numerical solution $z(t)$. 

As can be seen from the right figure of Figure \ref{fig5}, similar to the left figure, the curve corresponding to the ratio of error and stepsize, denoted as $e_{strong}/\Delta^8$ with $e_{strong}(t):=\mathbb{E}(\|x(t)-z(t)\|^{16})$, fluctuates within a certain range, consistent with the theoretical results.

\end{example}

\appendix
\section{Proof of Theorem \ref{T--2}}

\begin{proof}
    First, fixed $T:=2\tau+ (2\log(2M_1))/M_2$. Due to the fourth condition of Condition \ref{con1}, we know that for any $t \in [0, 2T]$
    \begin{align}\label{T22-1}
        \mathbb{E}\Big(x(t;0,\xi)-X(t;0,\xi)\Big) \leq C_T^1\Big(1+\sup \limits_{-\tau\leq t \leq 0}\mathbb{E}(|X(t;0,\xi)|^r)\Big)\Delta^q \leq C_T^1\Big(1+\mathbb{E}\|\xi\|^r\Big)\Delta^q,
    \end{align}
where $C_T^1$ is a positive constant dependent on $2T$. As for any $t_k \in [T, 3T]$, recalling that $\{x_{t_k}\}_{k \in \mathbb{N}}$ is a time-homogeneous Markov chain, and $T$ is $N$ times of $\Delta$, thus, for any $B \in \mathfrak{B}(C)$, $\xi \in C$, and $i \in \mathbb{N}_+$
  \begin{align}
      \mathbb{P}\big(x_{t_{k+i}} \in B ~\big|~x_{t_k} = \xi  \big)= \mathbb{P}\big(~x_{t_i} \in B ~\big|~x_{0}= \xi \big),
  \end{align}
  This implies that for any Borel set $\widehat{B} \in \mathbb{R}^d$ and $t \in [T, 3T]$,
\begin{align*}
    \mathbb{P}\big(x(t) \in \widehat{B}~\big| ~x_{T} = \xi  \big)= \mathbb{P}\big(x(t-T) \in \widehat{B}~\big| ~x_{0} = \xi  \big)
\end{align*}
Meanwhile, note that both the true and numerical solutions have the following flow property as, e.g., in \cite{Mao2007} or \cite[p.78]{MY2006} 
\begin{align}
    x(t;0,\xi)= x(t;s,x_s^{0,\xi}),\quad X(t;0,\xi)= X(t;s,X_s^{0,\xi}) \quad \quad \forall~ 0 \leq s < t < \infty,
\end{align}
 provided $s$ is the multiple of $\Delta$. Thus, by using the second and fourth conditions of Condition \ref{con1} along with the Young inequality, it is not difficult to see that for any $t \in [T, 3T]$
\begin{align}
    &\quad\mathbb{E}\Big(|x(t;0,\xi)-X(t;0,\xi)|^p\Big) \\ \notag
    & \leq 2^p \mathbb{E}\Big(|x(t;0,\xi)-x(t;T, X_T^{0,\xi})|^p\Big)+ 2^p\mathbb{E}\Big(|x(t;T,X_T^{0,\xi})-X(t;0, \xi)|^p\Big) \\ \notag
    & = 2^p\mathbb{E}\Big(|x(t;T, x_T^{0,\xi})-x(t;T, X_T^{0,\xi})|^p\Big)+2^p\mathbb{E}\Big(|x(t;T,X_T^{0,\xi})-X(t;T,X_T^{0,\xi})|^p\Big) \\ \notag
    & =2^p \mathbb{E}\Big(|x(t-T;0,x_T^{0,\xi})-x(t-T;0,X_T^{0,\xi})|^p\Big) +2^p\mathbb{E}\Big(|x(t;T,X_T^{0,\xi})-X(t;T,X_T^{0,\xi})|^p\Big) \\ \notag
    &\leq 2^p M_2\Big(\sup \limits_{-\tau\leq s \leq 0}\mathbb{E}(|x(T-s;0,\xi)-X(T-s;0,\xi)|^p)\Big)+ 2^p\mathbb{E}\Big(|x(t;T,X_T^{0,\xi})-X(t;T,X_T^{0,\xi})|^p\Big) \\ \notag
    &= 2^p M_2 \Big(\sup \limits_{T-\tau \leq s \leq T}\mathbb{E}(|x(t;0,\xi)-X(t;0,\xi)|^p)\Big)e^{-M_3(t-T)}+ 2^p C_T^1\Big(1+ \sup\limits_{T-\tau\leq t \leq T }\mathbb{E}(|X(t;0,\xi)|^r)\Big)\Delta^q.
\end{align}

From this, we can derive that for any $t_k \in [2T,3T]$
\begin{align}\label{T22-2}
      &\quad\mathbb{E}\Big(|x(t;0,\xi)-X(t;0,\xi)|^p\Big) \\ \notag
       &\leq 2^p M_2\Big(\sup \limits_{T-\tau \leq s \leq T}\mathbb{E}(|x(t;0,\xi)-X(t;0,\xi)|^p)\Big)e^{-M_3T}+ 2^p C_T^1\Big(1+ \sup\limits_{T-\tau\leq t \leq T }\mathbb{E}(|X(t;0,\xi)|^r)\Big)\Delta^q.
\end{align}
By the definition of $T$, we have
\begin{align}
    2^p M_2e^{-M_3T} \leq e^{\frac{-M_3T}{2}}
\end{align}
Thus, combining this with the second condition of Condition \ref{con1} and \eqref{T22-2} yields 
\begin{align}\label{T22-3}
    \mathbb{E}\Big(|x(t;0,\xi)-X(t;0,\xi)|^p\Big)\leq \Big(\sup \limits_{T-\tau \leq s \leq T}\mathbb{E}(|x(t;0,\xi)-X(t;0,\xi)|^p)\Big)e^{\frac{-M_3T}{2}}+ K_T\Delta^q,
\end{align}
where $K_{T}:=2 C_{T}^1(1+M_1)$. Then for any $t \in [3T, 4T]$, we can repeat the derivation of \eqref{T22-3} such that
\begin{align}\label{T22-4}
   &\quad\mathbb{E}\Big(|x(t;0,\xi)-X(t;0,\xi)|^p\Big) \\ \notag
   &= 2^p\mathbb{E}\Big(|x(t;0,\xi)-x(t;2T,X_{2T}^{0,\xi})|^p\Big)+ 2^p\mathbb{E}\Big(|x(t;2T,X_{2T}^{0,\xi})-X(t;0, \xi)|^p\Big) \\ \notag
    & = 2^p\mathbb{E}\Big(|x(t;2T, x_{2T}^{0,\xi})-x(t;2T, X_{2T}^{0,\xi})|^p\Big)+2^p\mathbb{E}\Big(|x(t;2T,X_{2T}^{0,\xi})-X(t;2T,X_{2T}^{0,\xi})|^p\Big) \\ \notag
    &\leq 2^p M_2 \Big(\sup \limits_{2T-\tau \leq s \leq 2T}\mathbb{E}(|x(t;0,\xi)-X(t;0,\xi)|^p)\Big)e^{-M_3(t-2T)}+ 2^p C_T^1\Big(1+ \sup\limits_{2T-\tau\leq t \leq 2T }\mathbb{E}(|X(t;0,\xi)|^r)\Big)\Delta^q \\ \notag
    & \leq \Big(\sup \limits_{2T-\tau \leq s \leq 2T}\mathbb{E}(|x(t;0,\xi)-X(t;0,\xi)|^p)\Big)e^{\frac{-M_3T}{2}}+ K_T\Delta^q.
\end{align}

Continuing this approach, for any $t_k \geq 0$, choose $l \in \mathbb{N}_{+}$ such that $t_k \in [(l+1)T, (l+2)T]$, then we have
\begin{align}\label{T22-5}
  &\quad  \mathbb{E}\Big(|x(t;0,\xi)-X(t;0,\xi)|^p\Big)\\ \notag
    &\leq \Big(\sup \limits_{lT-\tau \leq s \leq lT}\mathbb{E}(|x(t;0,\xi)-X(t;0,\xi)|^p)\Big)e^{\frac{-M_3T}{2}}+ K_T\Delta^q \\ \notag
   & \leq \Big(\sup \limits_{(l-2)T-\tau \leq s \leq (l-2)T}\mathbb{E}(|x(t;0,\xi)-X(t;0,\xi)|^p)\Big)e^{-M_3T}+ K_T\Big(1+e^{\frac{-M_3T}{2}}\Big)\Delta^q \\ \notag
    &\quad \vdots \\ \notag
     &\leq \Big(\sup \limits_{0 \leq s \leq 2T}\mathbb{E}(|x(t;0,\xi)-X(t;0,\xi)|^p)\Big)e^{-\frac{\lfloor 0.5(l+1)\rfloor M_3T}{2}} + K_T\Big(\sum \limits_{j=0}^{\lfloor 0.5(l+1)-1\rfloor}e^{-\frac{jM_3T}{2}}\Big)\Delta^q.
\end{align}

Obviously,
\begin{align}\label{T22-6}
     \sum \limits_{j=0}^{\infty}e^{-\frac{jM_3T}{2}} \leq \frac{e^{\frac{M_3T}{2}}}{e^{\frac{M_3T}{2}}-1}
\end{align}

Thus, combining \eqref{T22-1}, \eqref{T22-5} and \eqref{T22-6} together yields
\begin{align}
     \mathbb{E}\Big(|x(t;0,\xi)-X(t;0,\xi)|^p\Big)\leq \Big(C_T^1(1+\mathbb{E}(\|\xi\|^r))+ K_T\frac{e^{\frac{M_3T}{2}}}{e^{\frac{M_3T}{2}}-1}\Big)\Delta^q.
\end{align}

Since $\xi \in \mathcal C_{\mathcal F_0}^{p}$, let 
\begin{align*}
    C:=C_T^1(1+\mathbb{E}(\|\xi\|^r))+ K_T\frac{e^{\frac{M_3T}{2}}}{e^{\frac{M_3T}{2}}-1}.
\end{align*}
The required assertion follows.
\end{proof}

\section{Proof of Lemma \ref{L13} }
\begin{proof}
    First, fixed $T = 4\tau + (4\log (2K_{{1}}))/K_{{2}}$. Then, by the elementary inequality $|a+b|^2 \leq 2(a^2+b^2)$, along with the third and fourth conditions of Condition \ref{con19}, we can see that for any $t_k \in [0, 2T]$
    \begin{align}\label{L-16-1}
        \mathbb{E} \|X_{t_k}^{0,\xi}\|^2
        &\leq 2 \mathbb{E}\Big(\|X_{t_k}^{0,\xi}-x_{t_k}^{0,\xi}\|^2\Big)+2 \mathbb{E} \|x_{t_k}^{0,\xi} \|^2  \\ \notag
        &\leq 2C_T^1(1+\mathbb{E}\|\xi\|^2)\Delta^q+2 K_{{1}}\mathbb{E} \|\xi \|^2 e^{-K_{{2}} t_k} +2K_{{3}} \\ \notag
        &\leq 2 K_{{1}}\mathbb{E} \|\xi \|^2 e^{-K_{{2}} t_k} +2C_T^1\mathbb{E}\|\xi\|^2\Delta^q+ 2\Big(K_{{3}}+C_{T}^1\Delta^q\Big).
    \end{align}
   where $C_T^2$ is a positive constant dependent on $2T$. Next, similar to the proof of Theorem \ref{T--2}, for any $t_k \in [T,3T]$, by applying the first, third and fourth conditions of Condition \ref{con19}, along with the flow property, we can also derive that
    \begin{align*}
         \mathbb{E} \|x_{t_k}^{T, X_{T}^{0,\xi}} \|^2 &\leq K_{{1}} \mathbb{E} \|X_{T}^{0,\xi} \|^2 e^{-K_{{2}} (t_k-T)} + K_{{3}} , 
    \end{align*}
    and
    \begin{align*}
         \mathbb{E}\Big(\|X_{t_k}^{0,\xi}-x_{t_k}^{T, X_{T}^{0,\xi}}\|^2\Big)=\mathbb{E}\Big(\|X_{t_k}^{T, X_{T}^{0,\xi}}-x_{t_k}^{T, X_{T}^{0,\xi}}\|^2\Big) \leq  C_T^2 \Big(\sup \limits_{ T-\tau\leq t_k \leq T}\mathbb{E}\|X_{t_k}^{0,\xi}\|^2 +1\Big)\Delta^q.
    \end{align*}
   Thus, for any $t_k \in [T, 3T]$
    \begin{align}\label{L-16-2}
        \mathbb{E} \|X_{t_k}^{0,\xi}\|^2 &\leq 2  \mathbb{E}\Big(\|X_{t_k}^{0,\xi}-x_{t_k}^{T, X_{T}^{0,\xi}}\|^2\Big)+2  \mathbb{E} \|x_{t_k}^{X_{T}^{0,\xi}} \|^2  \\ \notag
        &\leq 2K_{{1}}e^{-K_{{2}} (t_k-T)}\mathbb{E} \|X_{T}^{0,\xi} \|^2+2  C_T^2\Delta^q\Big(\sup \limits_{ T-\tau\leq t_k \leq T}\mathbb{E}\|X_{t_k}^{0,\xi}\|^2 \Big)+ 2\Big(K_{{3}}+ C_T^2\Delta^q\Big).
    \end{align}
  This implies that for any $t_k \in [2T,3T]$
    \begin{align}\label{L-16-3}
        \sup \limits_{2T \leq  t_k \leq 3T} \mathbb{E} \|X_{t_k}^{0,\xi}\|^2 \leq 2\Big(K_{{1}}e^{-K_{{2}} T}+  C_T^2\Delta^q\Big) \Big(\sup \limits_{0 \leq  t_k \leq T} \mathbb{E} \|X_{t_k}^{0,\xi}\|^2 \Big)+ 2\Big(K_{{3}}+ C_T^2\Delta^q\Big).
    \end{align}
    Since $T= 4\tau + (4\log (2K_{{1}}))/K_{{2}}$ and
    \begin{align*}
        \lim \limits_{\Delta \to 0} \Big(2K_{{1}}e^{-K_{{2}} T}+ 2 C_T^2\Delta^q\Big) = \lim \limits_{\Delta \to 0} \Big(e^{-\frac{3K_{{2}}T}{4}} +  C_T^2\Delta^q\Big) =e^{-\frac{3K_{{2}}T}{4}}.
    \end{align*}
   There exists a sufficiently small $\Delta^* \in (0,1)$ such that for any $\Delta \in (0, \Delta^*)$
    \begin{align}\label{L-16-4}
        2\Big(K_{{1}}e^{-K_{{2}} T}+  C_T^2\Delta^q\Big) \leq e^{-\frac{K_{{2}}T}{2}}.
    \end{align}
    Inserting \eqref{L-16-4} into \eqref{L-16-3} yields
    \begin{align*}
          \sup \limits_{2T \leq t_k \leq 3T} \mathbb{E} \|X_{t_k}^{0,\xi}\|^2 &\leq \Big(\sup \limits_{T-\tau \leq t_k \leq T} \mathbb{E} \|X_{t_k}^{0,\xi}\|^2 \Big)e^{-\frac{K_{{2}}T}{2}} + 2\Big(K_{{3}}+ C_T^2\Delta^q\Big)\\ \notag
          &\leq \Big(\sup \limits_{0 \leq t_k \leq T} \mathbb{E} \|X_{t_k}^{0,\xi}\|^2 \Big)e^{-\frac{K_{{2}}T}{2}} + 2\Big(K_{{3}}+ C_T^2\Delta^q\Big).
    \end{align*}
    Then, for any $t_{N+k} \in [3T, 4T]$, since $\{X_{t_k}^{0, \xi}\}_{k \in \mathbb{N}}$ is a homogeneous Markov chain, similar to the previous derivation, we have
    \begin{align*}
         \sup \limits_{3T \leq t_{N+k} \leq 4T} \mathbb{E} \|X_{t_{N+k}}^{0, \xi}\|^2 &=   \sup \limits_{2T \leq t_k \leq 3T} \mathbb{E} \|X_{t_{k}}^{0, X_{t_N}^{0,\xi}}\|^2  \\ \notag
         &\leq e^{-\frac{K_{{2}}T}{2}} \Big( \sup \limits_{0 \leq t_k \leq T} \mathbb{E} \|X_{t_{k}}^{0, X_{t_N}^{0,\xi}}\|^2 \Big)+ 2\Big(K_{{3}}+ C_T^2\Delta^q\Big) \\ \notag
         &= e^{-\frac{K_{{2}}T}{2}} \Big(\sup \limits_{T \leq t_k \leq 2T} \mathbb{E} \|X_{t_k}^{0,\xi}\|^2 \Big)+ 2\Big(K_{{3}}+ C_T^2\Delta^q\Big).
    \end{align*}
   Continuing this approach, it is not difficult to see that for any  $i \in \mathbb{N}_{+}$
    \begin{align}\label{L-16-6}
         \sup \limits_{(i+1))T \leq t_{k} \leq (i+2)T} \mathbb{E} \|X_{t_{k}}^{0,\xi}\|^2
          &\leq \Big(\sup \limits_{(i-1)T \leq t_k \leq iT} \mathbb{E} \|X_{t_k}^{0,\xi}\|^2 \Big)e^{-\frac{K_{{2}}T}{2}} + 2\Big(K_{{3}}+ C_T^2\Delta^q\Big) \\ \notag
          &\quad \vdots \\ \notag
          &\leq \Big(\sup \limits_{ 0\leq t_k \leq 2T} \mathbb{E} \|X_{t_k}^{0,\xi}\|^2 \Big)e^{-\frac{\lfloor0.5(i+1)\rfloor K_{{2}}T}{2}} + 2\Big(K_{{3}}+ C_T^2\Delta^q\Big)\Big(\sum \limits_{j=0}^{\lfloor0.5(i+1)\rfloor} e^{-\frac{j K_{{2}}T}{2}}\Big).
    \end{align}
    Note that 
    \begin{align}\label{L-16-7}
        \sum \limits_{j=0}^{\infty} e^{-\frac{j K_{{2}}T}{2}}\leq e^{\frac{K_{{2}}T}{2}}/(e^{\frac{K_{{2}}T}{2}}-1).
    \end{align}
    Thus, combining \eqref{L-16-1}, \eqref{L-16-6} and \eqref{L-16-7} together and setting
    \begin{align*}
        K_6&:=2K_{{1}}+2C_T^2, \\ \notag
        K_7&:=2 \Big(K_{{3}}+C_T^2\Big)\Big(2+\frac{e^{\frac{K_{{2}}T}{2}}}{e^{\frac{K_{{2}}T}{2}}-1}\Big),
    \end{align*}
    the desired assertion follows.
\end{proof}

\bibliography{References}

\begin{thebibliography}{10}

\bibitem{ACO2023}
L.~Angeli, D.~Crisan, and M.~Ottobre.
\newblock Uniform in time convergence of numerical schemes for stochastic
  differential equations via strong exponential stability: Euler methods,
  split-step and tamed schemes.
\newblock {\em arXiv:2303.15463}, 2023.

\bibitem{AHMG2007}
M.~Arriojas, Y.~Hu, S.-E. Mohammed, and G.~Pap.
\newblock A delayed black and scholes formula.
\newblock {\em Stochastic Anal. Appl.}, 25(2):471 – 492, 2007.

\bibitem{BSY-2023}
J.~Bao, J.~Shao, and C.~Yuan.
\newblock Invariant probability measures for path-dependent random diffusions.
\newblock {\em Nonlinear Anal.}, 228, 2023.

\bibitem{Buckwar2000}
E.~Buckwar.
\newblock Introduction to the numerical analysis of stochastic delay
  differential equations.
\newblock {\em J. Comput. Appl. Math.}, 125(1-2):297 – 307, 2000.

\bibitem{CGMP2019}
F.~Cacace, A.~Germani, C.~Manes, and M.~Papi.
\newblock Predictor-based control of stochastic systems with nonlinear
  diffusions and input delay.
\newblock {\em Automatica}, 107:43 – 51, 2019.

\bibitem{CKBW2015}
Y.~Cai, Y.~Kang, M.~Banerjee, and W.~Wang.
\newblock A stochastic sirs epidemic model with infectious force under
  intervention strategies.
\newblock {\em J. Differential Equations}, 259(12):7463 -- 7502, 2015.

\bibitem{CDO2021}
D.~Crisan, P.~Dobson, and M.~Ottobre.
\newblock Uniform in time estimates for the weak error of the {E}uler method
  for {SDE}s and a pathwise approach to derivative estimates for diffusion
  semigroups.
\newblock {\em Trans. Amer. Math. Soc.}, 374(5):3289--3330, 2021.

\bibitem{DGM2021}
J.~P. Décamps, F.~Gensbittel, and T.~Mariotti.
\newblock Investment timing and technological breakthroughs, 2021.

\bibitem{DNY2016}
N.H. Du, D.H. Nguyen, and G.G. Yin.
\newblock Conditions for permanence and ergodicity of certain stochastic
  predator prey models.
\newblock {\em J. Appl. Probab.}, 53(1):187 -- 202, 2016.

\bibitem{GLL2023}
S.~Gao, X.~Li, and Z.~Liu.
\newblock Stationary distribution of the milstein scheme for stochastic
  differential delay equations with first-order convergence.
\newblock {\em Appl. Math. Comput.}, 458, 2023.

\bibitem{GA1997}
S.~Grenadier and A.~Weiss.
\newblock Investment in technological innovations: An option pricing approach.
\newblock {\em Journal of Financial Economics}, 44(3):397--416, 1997.

\bibitem{GMY2018}
Q.~Guo, X.~Mao, and R.~Yue.
\newblock The truncated euler–maruyama method for stochastic differential
  delay equations.
\newblock {\em Numer. Algorithms.}, 78(2):599 – 624, 2018.

\bibitem{HMS2003}
D.~J. Higham, X.~Mao, and A.~M. Stuart.
\newblock Exponential mean-square stability of numerical solutions to
  stochastic differential equations.
\newblock {\em LMS J. Comput. Math.}, 6:297 – 313, 2003.

\bibitem{Hu1996}
Yaozhong Hu.
\newblock Semi-implicit euler-maruyama scheme for stiff stochastic equations.
\newblock In {\em Stochastic analysis and related topics, {V} ({S}ilivri,
  1994)}, volume~38 of {\em Progr. Probab.}, pages 183--202. Birkh\"{a}user
  Boston, Boston, MA, 1996.

\bibitem{IW1989}
N.~Ikeda and S.~Watanabe.
\newblock {\em Stochastic differential equations and diffusion processes}.
\newblock North-Holland, Amsterdam, 1989.

\bibitem{JY2017}
Y.~Ji and C.~Yuan.
\newblock Tamed em scheme of neutral stochastic differential delay equations.
\newblock {\em J. Comput. Appl. Math.}, 326:337 – 357, 2017.

\bibitem{LM2006}
J.~Lei and M.~C. Mackey.
\newblock Stochastic differential delay equation, moment stability, and
  application to hematopoietic stem cell regulation system.
\newblock {\em SIAM J. Appl. Math.}, 67(2):387 – 407, 2006.

\bibitem{LMYY2018}
X.~Li, Q.~Ma, H.~Yang, and C.~Yuan.
\newblock The numerical invariant measure of stochastic differential equations
  with markovian switching.
\newblock {\em SIAM J. Numer. Anal.}, 56(3):1435 -- 1455, 2018.

\bibitem{LM2009}
X.~Li and X.~Mao.
\newblock Population dynamical behavior of non-autonomous lotka-volterra
  competitive system with random perturbation.
\newblock {\em Discrete Contin. Dynam. Systems}, 24(2):523 -- 545, 2009.

\bibitem{LMS2023}
X.~Li, X.~Mao, and G.~Song.
\newblock Explicit approximation of invariant measure for stochastic delay
  differential equations with the nonlinear diffusion term.
\newblock {\em J. Theoret. Probab.}, 2023.

\bibitem{LMY2019}
X.~Li, X.~Mao, and G.~Yin.
\newblock Explicit numerical approximations for stochastic differential
  equations in finite and infinite horizons: Truncation methods, convergence in
  pth moment and stability.
\newblock {\em IMA J. Numer. Anal.}, 39(2):847--892, 2019.

\bibitem{LCF2004}
M.~Liu, W.~Cao, and Z.~Fan.
\newblock Convergence and stability of the semi-implicit euler method for a
  linear stochastic differential delay equation.
\newblock {\em J. Comput. Appl. Math.}, 171(1-2):255 – 268, 2004.

\bibitem{Mao2007}
X.~Mao.
\newblock Exponential stability of equidistant euler-maruyama approximations of
  stochastic differential delay equations.
\newblock {\em J. Comput. Appl. Math.}, 200(1):297 – 316, 2007.

\bibitem{M2007}
X.~Mao.
\newblock {\em Stochastic Differential Equations and Applications}.
\newblock Horwood, Chichester, UK, 2 edition, 2007.

\bibitem{Mao2015}
X.~Mao.
\newblock Almost sure exponential stability in the numerical simulation of
  stochastic different equations.
\newblock {\em SIAM J. Numer. Anal.}, 53(1):370 – 389, 2015.

\bibitem{MR2005}
X.~Mao and M.~J. Rassias.
\newblock Khasminskii-type theorems for stochastic differential delay
  equations.
\newblock {\em Stochastic Anal. Appl.}, 23(5):1045 – 1069, 2005.

\bibitem{MS2003}
X.~Mao and S.~Sabanis.
\newblock Numerical solutions of stochastic differential delay equations under
  local lipschitz condition.
\newblock {\em J. Comput. Appl. Math.}, 151(1):215 – 227, 2003.

\bibitem{MY2006}
X.~Mao and C.~Yuan.
\newblock {\em Stochastic Differential Equations with Markovian Switching}.
\newblock Imperial College Press, London, 2 edition, 2006.

\bibitem{M1974}
E.~J. McShane.
\newblock {\em Stochastic Calculus and Stochastic Models. Academic}.
\newblock Academic Press, 1974.

\bibitem{MPS1998}
G.N. Milstein, E.~Platen, and H.~Schurz.
\newblock Balanced implicit methods for stiff stochastic systems.
\newblock {\em SIAM J. Numer. Anal.}, 35(3):1010 -- 1019, 1998.

\bibitem{M1986}
S.E.A. Mohammed.
\newblock {\em Stochastic Functional Differential Equations}.
\newblock Longman, New York, 1986.

\bibitem{RRV2006}
M.~Reiß, M.~Riedle, and O.~van Gaans.
\newblock Delay differential equations driven by l\'{e}vy processes:
  stationarity and feller properties.
\newblock {\em Stochastic Process. Appl.}, 116(10):1409 --1432, 2006.

\bibitem{RVB2010}
P.~Rué, J.~Villà-Freixa, and K.~Burrage.
\newblock Simulation methods with extended stability for stiff biochemical
  kinetics.
\newblock {\em BMC Syst. Biol. 4 (110) (2010) 1–13.}, 4(110):1 -- 13, 2010.

\bibitem{SHGL2022}
G.~Song, J.~Hu, S.~Gao, and X.~Li.
\newblock The strong convergence and stability of explicit approximations for
  nonlinear stochastic delay differential equations.
\newblock {\em Numer. Algorithms}, 89(2):855 – 883, 2022.

\bibitem{WWM2019}
Y.~Wang, F.~Wu, and X.~Mao.
\newblock Stability in distribution of stochastic functional differential
  equations.
\newblock {\em Systems Control Lett.}, 132, 2019.

\bibitem{YM2003}
C.~Yuan and X.~Mao.
\newblock Asymptotic stability in distribution of stochastic differential
  equations with markovian switching.
\newblock {\em Stochastic Process. Appl.}, 103(2):277 -- 291, 2003.

\bibitem{YM2004}
C.~Yuan and X.~Mao.
\newblock Stability in distribution of numerical solutions for stochastic
  differential equations.
\newblock {\em Stochastic Anal. Appl.}, 22(5):1133 -- 1150, 2004.

\bibitem{YZM2003}
C.~Yuan, J.~Zou, and X.~Mao.
\newblock Stability in distribution of stochastic differential delay equations
  with markovian switching.
\newblock {\em Systems Control Lett.}, 50(3):195 – 207, 2003.

\bibitem{Zhou2015}
S.~Zhou.
\newblock Strong convergence and stability of backward euler–maruyama scheme
  for highly nonlinear hybrid stochastic differential delay equation.
\newblock {\em Calcolo}, 52(4):445 – 473, 2015.

\end{thebibliography}
\end{document}